\documentclass[11pt,Bold]{mcgilletdclass}
\usepackage[dvips,final]{graphicx}
\usepackage[dvips]{geometry}
\usepackage{epigraph}

\usepackage{amsmath,amsthm, amssymb, verbatim,setspace,longtable,tabularx,enumitem}
\usepackage{tikz}
\usepackage[multiple]{footmisc}

\newtheorem{thm}[equation]{Theorem}
\newtheorem{lem}[equation]{Lemma}
\newtheorem{prop}[equation]{Proposition}
\newtheorem{conj}[equation]{Conjecture}
\newtheorem{claim}[equation]{Claim}
\newtheorem{cor}[equation]{Corollary}
\newtheorem{ques}[equation]{Question}
\newtheorem{prob}[equation]{Problem}

\newtheorem*{CML}{The Colour Matching Lemma}
\newtheorem*{DP}{The Dinitz Problem}
\newtheorem*{GT}{Galvin's Theorem}
\newtheorem*{KT}{Kahn's Theorem}
\newtheorem*{LCC}{The List Colouring Conjecture}
\newtheorem*{GMC}{Gravier and Maffray's Conjecture}
\newtheorem*{LTCC}{The List Total Colouring Conjecture}

\newtheorem{innercustomconj}{Conjecture}
\newenvironment{customconj}[1]
  {\renewcommand\theinnercustomconj{#1}\innercustomconj}
  {\endinnercustomconj}
  
  \newtheorem{innercustomques}{Question}
\newenvironment{customques}[1]
  {\renewcommand\theinnercustomques{#1}\innercustomques}
  {\endinnercustomques}

\newtheorem*{LSCC}{The List Square Colouring Conjecture}
\newtheorem*{OC}{Ohba's Conjecture}
\newtheorem*{OLOC}{On-Line Ohba's Conjecture}
\newtheorem*{T5CT}{Thomassen's Five Colour Theorem}
\newtheorem*{ATT}{Alon and Tarsi's Theorem}
\newtheorem*{HT}{Hall's Theorem}
\newtheorem{lemma}{Lemma}

\theoremstyle{definition}
\newtheorem{defn}[equation]{Definition}
\newtheorem{obs}[equation]{Observation}

\newtheorem{note}[equation]{Note}
\newtheorem{rem}[equation]{Remark}
\newtheorem{phase}{Phase}

\theoremstyle{remark}

\newtheoremstyle{case}{}{}{\normalfont}{}{\itshape}{\normalfont:}{ }{}

\theoremstyle{case}
\newtheorem{case}{Case}

\makeatletter
\newenvironment{chapquote}[2][4em]
  {\setlength{\@tempdima}{#1}%
   \def\chapquote@author{#2}%
   \parshape 1 \@tempdima \dimexpr\textwidth-1\@tempdima\relax%
   \itshape\small\noindent}
  {\par\normalfont\small\hfill\\\vspace{-0.15cm} \hfill---\ \chapquote@author\hspace*{\@tempdima}\par\bigskip}
\makeatother

\newcommand{\ch}{\text{\rm ch}}  
\newcommand{\OL}{\text{\rm ch}^{{\rm OL}}} 
\newcommand{\col}{\text{\rm col}}

\SetTitle{\huge{Choosability of Graphs with Bounded Order: Ohba's Conjecture and Beyond}}%
\SetAuthor{Jonathan A. Noel}%
\SetDegreeType{Master of Science}%
\SetDepartment{Department of Mathematics and Statistics}%
\SetUniversity{McGill University}%
\SetUniversityAddr{Montr\'{e}al, Qu\'{e}bec}%
\SetThesisDate{May 2013}%
\SetRequirements{A thesis submitted to McGill University in partial fulfilment of the
requirements of the degree of Master of Science.}%
\SetCopyright{\copyright Jonathan A. Noel, 2013}%

\makeatletter
\def\@part[#1]#2{%
  \ifnum \c@secnumdepth >-2\relax
    \refstepcounter{part}%
    \addcontentsline{toc}{part}{%
      {\rm Part \thepart: #1}}%
  \else
    \addcontentsline{toc}{part}{#1}%
  \fi
  \markboth{}{}%
  {
  
  \clearpage
		\pagestyle{plain}%
		\setboolean{SetDSpace}{false}%
    \GoSingle%
  \vspace*{0\p@}%
  {\parindent \z@ \position \reset@font
        {\Huge \scshape  }
        \par\nobreak
        \vspace*{10\p@}%
        \interlinepenalty\@M
        \thickhrule
        \par\nobreak
        \vspace*{0\p@}%
        {\hfill\huge \bfseries \partname\nobreakspace \thepart: #2 \par\nobreak}
        \par\nobreak
        \vspace*{0\p@}%
        \thickhrule
    \vskip 0\p@
    \setboolean{SetDSpace}{true}%
  }  

   }%
  \@endpart}
  
\def\@endpart{%
   \thispagestyle{empty}%
   \vfil\newpage
   \if@twoside
     \if@openright
       \null
       \thispagestyle{empty}%
       \newpage
     \fi
   \fi
   \if@tempswa
     \twocolumn   \fi}

\makeatother

\newcommand\immediateaddcontentsline[3]{%
        \begingroup
        \let\origwrite\write
        \def\write{\immediate\origwrite}
        \addcontentsline{#1}{#2}{#3}%
        \endgroup
}

\begin{document}

\maketitle%

\begin{romanPagenumber}{2}%

\SetAbstractEnName{Abstract}%
\SetAbstractEnText{The \emph{choice number} of a graph $G$, denoted $\ch(G)$, is the minimum integer $k$ such that for any assignment of lists of size $k$ to the vertices of $G$, there is a proper colouring of $G$ such that every vertex is mapped to a colour in its list. For general graphs, the choice number is not bounded above by a function of the chromatic number.
\linebreak\vspace{-0.28cm}\linebreak In this thesis, we prove a conjecture of Ohba which asserts that $\ch(G)=\chi(G)$ whenever $|V(G)|\leq 2\chi(G)+1$. We also prove a strengthening of Ohba's Conjecture which is best possible for graphs on at most $3\chi(G)$ vertices, and pose several conjectures related to our work.}
\AbstractEn%

\SetAbstractFrName{Abr\'{e}g\'{e}}%
\SetAbstractFrText{Le \emph{nombre de choix} d'un graphe $G$, not\'{e} $\text{ch}(G)$, est le plus petit entier $k$ tel que pour toute affectation de listes de taille $k$ au sommets de $G$, il y a une coloration de $G$ tel que chaque sommet de $G$ est color\'{e} par une couleur de sa liste. En g\'{e}n\'{e}ral, le nombre de choix n'est pas born\'{e} sup\'{e}rieurement par une fonction du nombre chromatique.\linebreak\vspace{-0.28cm}\linebreak
Dans cette th\`{e}se, nous d\'{e}montrons une conjecture de Ohba qui affirme que $\ch(G) = \chi(G)$ d\`{e}s que $|V(G)| \leq 2\chi(G)+1$. Nous d\'{e}montrons aussi une version plus forte de la conjecture de Ohba qui est optimale pour les graphes ayant au plus $3\chi(G)$ sommets, et \'{e}non\c{c}ons plusieurs conjectures par rapport \`{a} nos travaux.}%
\AbstractFr%

\SetDeclarationName{Declaration}%
\SetDeclarationText{This thesis contains no material which has been accepted in whole, or in part, for any other degree or diploma. Chapters 4 and 6 of this thesis contain new contributions to knowledge. The results of these chapters have been, or will be, submitted for publication in peer-reviewed journals. The result of Chapter 4 is based on joint work with Bruce A. Reed and Hehui Wu. The result of Chapter 6 is based on joint work with Douglas B. West, Hehui Wu, and Xuding Zhu.}%
\Declaration%

\SetAcknowledgeName{Acknowledgements}%
\SetAcknowledgeText{Paul Erd\H{o}s often spoke about the importance of having an `open brain.' I certainly did not know the full meaning of this before I met my supervisor, Bruce Reed. Working with Bruce has been a `brain opening' experience on so many levels. I am forever grateful to Bruce for his support and generosity. Thank you for teaching me, inspiring me, and for always keeping me on my toes. 
\linebreak\vspace{-0.28cm}\linebreak  During my time at McGill, I have been very fortunate to be surrounded by a great group of friends. I will remember fondly the times that we spent eating together, drinking together, laughing together, and working together (often at the same time). 
\linebreak\vspace{-0.28cm}\linebreak Thanks to Cai and Lena for the philosophical discussions, and for feeding me a zillion times.  Thanks to Omkar and Yuting for unintentionally pushing me to run faster and to work harder (respectively). I forgive Cathryn for always stealing my desk. I would especially like to thank Liana, Laura and Mash for giving me so much moral support, and for always believing in me. 
\linebreak\vspace{-0.28cm}\linebreak The results of this thesis would not have been obtained without the contributions of great collaborators. Thanks to Bruce, Hehui, Doug and Xuding for sharing ideas during many illuminating discussions. 
\linebreak\vspace{-0.28cm}\linebreak I have been very fortunate to have had the opportunity to meet and work with many great researchers. Thanks to all of the people who invited or hosted me: Robert Morris in Rio de Janeiro; Ken-ichi Kawarabayashi and Kenta Ozeki in Tokyo; Katsuhiro Ota in Yokohama; David Avis in Kyoto; Bo Zhou in Guangzhou; Doug West and Xuding Zhu in Jinhua; Andrew King and Luis Goddyn in Burnaby.
\linebreak\vspace{-0.28cm}\linebreak I would also like to thank the other professors, postdocs, and students in the discrete mathematics group at McGill for maintaining an enjoyable and stimulating academic environment. Special thanks to those professors whose courses I attended and thoroughly enjoyed: Louigi, Adrian, Sergey, Linan and, of course, Bruce. 
\linebreak\vspace{-0.28cm}\linebreak It wasn't easy to leave Kamloops behind for the big francophone city out east. Thanks to everyone back home for continuing to root for me as I make the transition to (yet another) degree. Special thanks to Rick Brewster for introducing me to the world of graph theory research as an undergraduate. Last, but not least, I would like to thank my family for their constant love and support.}%
\Acknowledge%

\TOCHeading{Table of Contents}%
\LOTHeading{List of Tables}%
\LOFHeading{List of Figures}%

\tableofcontents %
\listoffigures %

\end{romanPagenumber}


{

{
\if@openright
    \cleardoublepage
  \else
    \clearpage
  \fi
\let\newpage\relax
\part{Preliminaries}
\label{prelim}
}

\let\newpage\relax
\let\clearpage\relax

{\addcontentsline{toc}{part}{Part~\ref{prelim}: Preliminaries}
\chapter{Introduction}
\label{intro}

\begin{chapquote}{Gian-Carlo Rota}
We often hear that mathematics consists mainly of ``proving theorems.'' Is a writer's job mainly that of ``writing sentences?''\end{chapquote}
\doublespacing

Some of the most fruitful areas of contemporary mathematics have been motivated by problems which are simple to state, yet disproportionately difficult to solve. One of the most famous such problems was known as the Four Colour Conjecture, first posed by Francis Guthrie in 1852 (see \cite{Update}), which asserts that the countries of any planar map can be coloured using at most four colours in such a way that neighbouring countries are assigned to different colours. Although the problem itself can be easily understood by most, its solution resisted the attempts of serious mathematicians for over 120 years until Appel and Haken finally solved it~\cite{AppelHaken1}. The problem can be naturally reformulated in terms of properly colouring the vertices of a planar graph using at most four colours and, because of this, the Four Colour Conjecture was instrumental in bringing attention to graph theory and establishing it as an interesting subject in its own right.\footnote{For a list of standard graph theoretic notation and terminology, see the glossary.}

To this day, graph theory remains a subject that is driven by simple and elegant problems. While many of these problems, like the Four Colour Conjecture, are investigated primarily for their aesthetic appeal,\footnote{After all, how many cartographers were anxiously awaiting the resolution of the Four Colour Conjecture?} there are many others which have a deeper practical importance. One well-known example is the Strong Perfect Graph Conjecture of Berge~\cite{Berge}, which was inspired by the problem of efficiently transmitting information through a noisy channel without introducing errors (i.e. \emph{perfect} transmission). As it turns out, the problem of perfect transmission is related to the notion of a perfect graph; a graph $G$ is \emph{perfect} if for every induced subgraph $H$ of $G$, the chromatic number of $H$ is equal to the clique number of $H$. The Strong Perfect Graph Conjecture states that a graph is perfect if and only if it does not contain an induced copy of an odd cycle of length at least five or its complement. After being open for more than 40 years, it was finally proved in a monumental paper of Chudnovsky, Roberston, Seymour and Thomas~\cite{Perfect}.

The main focus of this thesis is another tantalizing problem in graph theory, known as Ohba's Conjecture~\cite{OhbaOrig}. Ohba's Conjecture concerns a variant of graph colouring, known as \emph{choosability}, in which the objective is to find a proper colouring of a graph under the condition that the colour of each vertex $v$ of $G$ is contained in a particular list $L(v)$ of colours -- distinct vertices may have different lists. Such a colouring is called an \emph{acceptable} colouring. Specifically, Ohba's Conjecture says that if the order of $G$ is at most twice the chromatic number of $G$ plus one, then there exists an acceptable colouring whenever each vertex $v$ of $G$ has a list $L(v)$ of at least $\chi(G)$ colours to choose from.\footnote{Using standard graph theoretic notation, Ohba's Conjecture is the following: \emph{if $|V(G)|\leq 2\chi(G)+1$, then $\ch(G)= \chi(G)$.} }

The thesis is divided into three parts. Part~\ref{prelim} contains the relevant preliminary material, including a survey of some of the central results and conjectures in choosability. At the end of Chapter~\ref{graphchap}, we will begin to discuss Ohba's Conjecture within this broader context.

In Part~\ref{OhbaPart}, we focus our attention more specifically on Ohba's Conjecture itself. In Chapter~\ref{background}, we outline the history of the problem, including a summary of partial results. In Chapter~\ref{proof}, we give a detailed proof of Ohba's Conjecture from~\cite{NRW}, which is our main contribution. The proof is mainly composed of three parts: (a) a lemma which shows that, by applying Hall's Theorem, it is possible to obtain an acceptable colouring by modifying a specific type of non-acceptable colouring, (b) a method for constructing such a non-acceptable colouring under certain conditions regarding the distribution of colours in the lists, and (c) a proof that, if Ohba's Conjecture is false, then there exists a counterexample which meets these conditions. 

In Part~\ref{beyond} we look beyond Ohba's Conjecture, turning our focus to some related problems. In Chapter~\ref{glimpse} we highlight the relationship between Ohba's Conjecture and an old problem of Erd\H{o}s, Rubin and Taylor~\cite{ERT} on the choice number of complete multipartite graphs, and pose two conjectures based on this. We also discuss a generalization of Ohba's Conjecture to an `on-line' variant of choosability conjectured by Huang, Wong and Zhu~\cite{online}. In Chapter~\ref{strength} we provide a proof from~\cite{NWWZ} of a direct strengthening of Ohba's Conjecture which verifies some special cases of the conjectures posed in Chapter~\ref{glimpse}. Finally, we conclude the thesis in Chapter~\ref{concl} by summarizing some of the open problems discussed.

 
\section{Additional Remarks}

The Four Colour Conjecture has a very interesting history. Kempe published a ``solution'' in 1879 which stood for over a decade until a fatal oversight was uncovered by Heawood in 1890. Another proposed solution of Tait from 1880 was found to be incorrect by Petersen in 1891. However, each of these attempts turned out to be fruitful. Kempe's idea was modified by Heawood to prove the Five Colour Theorem (every planar graph can be properly coloured using at most five colours) and Tait's attempted proof revealed another formulation of the Four Colour Conjecture: \emph{every bridgeless cubic planar graph has a proper $3$-edge colouring.} 

When the full solution of Appel and Haken finally arrived in 1977, it was met with resistance. Their proof consisted of two parts: a rigorous mathematical argument, and a calculation which could only be reasonably checked by a computer. To some mathematicians, a proof which requires the use of a computer could not be trusted. Perhaps more troubling was that, even if the proof can be accepted as being complete and correct, it lacks the \emph{insight} of a proof that can be verified by hand. That is, by using a computer, we can prove that the Four Colour Conjecture is true, but we only gain a limited understanding of \emph{why} it is true. Later, another proof of the Four Colour Theorem was developed by Robertson, Sanders, Seymour and Thomas~\cite{4CTSeymour1}. Their proof is less complicated than Appel and Haken's, but the general framework is similar and it does not avoid the use of a computer. For an interesting book on the Four Colour Theorem, including its colourful history, see~\cite{Suffice}.

The Strong Perfect Graph Conjecture was one of two famous conjectures posed by Berge in~\cite{Berge}. The other, which was known as the Weak Perfect Graph Conjecture, states that a graph is perfect if and only if its complement is perfect. The Weak Perfect Graph Conjecture was first proved by Lov\'{a}sz in~\cite{WeakPerfect}. It is clear that the Strong Perfect Graph Conjecture implies the Weak Perfect Graph Conjecture, as their names would suggest. For more on perfect graphs and the Strong Perfect Graph Theorem see~\cite{PerfectBook} or~\cite{PerfectSurvey}.

There are many other problems which have inspired substantial progress in graph theory; some notable examples are as follows: Hadwiger's Conjecture~\cite{Hadwiger}, Tutte's $5$-Flow Conjecture~\cite{Tutte}, the Cycle Double Cover Conjecture of Szekeres and Seymour~\cite{Szekeres,Seymour}, the Erd\H{o}s-Faber-Lov\'{a}sz Conjecture~\cite{MostLike}, the Reconstruction Conjecture of Kelly and Ulam~\cite{Kelly,Ulam}, the Erd\H{o}s-Hajnal Conjecture~\cite{EH}, and the List Colouring Conjecture, the last of which will be discussed in Chapter~\ref{graphchap}. An individual deserving of special mention is Paul Erd\H{o}s, who thrived in the problem-driven nature of graph theory and contributed a plethora of lasting conjectures which continue to motivate the field; see~\cite{Chung} for a list. 

}
}
{
\newpage
\clearpage
\chapter{Choosability}
\label{graphchap}

\begin{chapquote}{Timothy Gowers, \emph{The two cultures of mathematics}}
One will not get anywhere in graph theory by sitting in an armchair and trying to understand graphs better. Neither is it particularly necessary to read much of the literature before tackling a problem: it is of course helpful to be aware of some of the most important techniques, but the interesting problems tend to be open precisely because the established techniques cannot easily be applied.\end{chapquote}

\doublespacing

\section{Definitions and Basic Properties}

Throughout this thesis, we follow standard graph theoretic notation and terminology, most of which can be found in~\cite{BondyMurtyNew}; for a summary, see the glossary. An instance of choosability is a graph $G$ in which every vertex $v$ is assigned to a list $L(v)$ of \emph{available} colours. The objective is to find an \emph{acceptable} colouring for $L$, which is a proper colouring $f$ of $G$ such that $f(v)\in L(v)$ for every vertex $v$ of $G$. We say that $G$ is \emph{$k$-choosable} if there exists an acceptable colouring for $L$ whenever $|L(v)|\geq k$ for all $v\in V(G)$. The \emph{choice number} of $G$, denoted $\ch(G)$, is defined as follows:
\[\ch(G):=\min\{k: G\text{ is }k\text{-choosable}\}.\]
It is important to notice that, in determining the choice number of $G$, it is the size of the \emph{smallest list} that is important, regardless of the \emph{total number} of colours $\left|\cup_{v\in V(G)}L(v)\right|$. 

Choosability was independently introduced by Vizing~\cite{Vizing} and Erd\H{o}s et al.~\cite{ERT}. It has since become one of the most popular variants of the classical graph colouring problem, and there are many well-read surveys on the topic; see, e.g.,~\cite{Alon,Krat,Tuza,Woodall}.

One attractive feature of choosability is that it arises naturally in many classical graph colouring problems. For example, given a subgraph $H$ of a graph $G$ and a proper colouring $f$ of $H$, one might wish to extend $f$ to a proper colouring of $G$ which maps into a set $C\supseteq f(V(H))$ of colours. This problem can be naturally formulated in terms of choosability. The goal here is to construct an acceptable colouring of $G-H$ in which each vertex $v\in V(G-H)$ is assigned to the list $L(v):=C-f\left(N(v)\cap V(H)\right)$.

It is clear that a graph $G$ is $k$-colourable if and only if there exists an acceptable colouring for $L$ when $L(v)=\{1,\dots,k\}$ for every vertex $v$ of $G$. By setting $k=\chi(G)-1$ in this example, we see that
\[\label{ch>=chi}\ch\geq \chi.\]

Initially, one may wonder if the choice number is actually \emph{equal} to the chromatic number in general. After all, if the vertices of $G$ are assigned to lists of size $\chi(G)$ which are not all the same, then 
\begin{enumerate}[leftmargin=*,labelindent=1em,label=$\bullet$]
\item some pairs of adjacent vertices may be assigned to lists which have a small intersection, and
\item in constructing an acceptable colouring, we are permitted to use more than $\chi(G)$ colours in total. 
\end{enumerate}
It would seem that, because of these two properties, assigning lists which are not identical could only make it \emph{easier} to find an acceptable colouring. However, this reasoning is flawed. As it turns out, a graph $G$ may not be $k$-choosable even when $k$ far exceeds the chromatic number of $G$.

A well-known example of this phenomenon comes from~\cite{ERT}. For $d\geq2$, let $C$ be a set of $2d-1$ distinct colours and let $L$ be a list assignment which assigns each $d$-subset of $C$ to exactly one vertex in each side of the bipartition of $K_{\binom{2d-1}{d},\binom{2d-1}{d}}$. The smallest such example, $K_{3,3}$, is illustrated in Figure~\ref{K3,3}.

\begin{figure}[htbp]
\begin{center}
\begin{tikzpicture}[label distance=3.5mm]
    \tikzstyle{every node}=[draw,circle,fill=black,minimum size=4pt,
                            inner sep=0pt]
    \draw (0,0) node (12L) [label=left:$\{{1,2}\}$] {}
        -- ++(0:2.75cm) node (12R) [label=right:$\{{1,2}\}$] {}
           ++(270:1.75cm) node (13R) [label=right:$\{{1,3}\}$] {}
        -- ++(180:2.75cm) node (13L) [label=left:$\{{1,3}\}$] {}
           ++(270:1.75cm) node (23L) [label=left:$\{{2,3}\}$] {}
        -- ++(0:2.75cm) node (23R) [label=right:$\{{2,3}\}$] {};
     \draw (12L)--(13R);
     \draw (12L)--(23R);
     \draw (13L)--(12R);
     \draw (13L)--(23R);
     \draw (23L)--(12R);
     \draw (23L)--(13R);
\end{tikzpicture}
\end{center}
\caption{\label{K3,3}A list assignment which demonstrates that $K_{3,3}$ is not $2$-choosable.}
\end{figure}
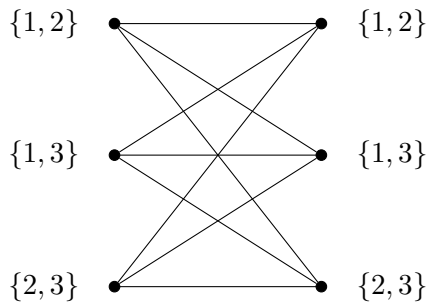 

We claim that there does not exist an acceptable colouring for $L$. Otherwise, let $f$ be such a colouring. We observe that for any set $S\subseteq C$ of at most $d-1$ colours, there is a vertex in each side of the bipartition whose list does not intersect $S$. Therefore, in order to map each vertex of  $K_{\binom{2d-1}{d},\binom{2d-1}{d}}$ to a colour in its list, we see that $f$ must map each side of the bipartition to a set of at least $d$ distinct colours. Moreover, since $f$ is a proper colouring, no colour can be used by $f$ on both sets of the bipartition. Thus, the image of $f$ must contain at least $2d$ distinct colours, contradicting the fact that $|C|=2d-1$. This argument proves the following:

\begin{prop}[Erd\H{o}s et al.~\cite{ERT}]
\label{d} 
For $d\geq2$, we have
\[\ch\left(K_{\binom{2d-1}{d},\binom{2d-1}{d}}\right)>d.\]
\end{prop}

As a corollary, we see that there is no general upper bound on the choice number in terms of the chromatic number.

\begin{cor}[Erd\H{o}s et al.~\cite{ERT}]
\label{nofunc}
For every function $g$ on $\mathbb{N}$, there exists a graph $G$ such that
\[\ch(G)>g(\chi(G)).\]
\end{cor}

\begin{proof}
Let $g$ be any function on $\mathbb{N}$ and define $d=\max\{g(2),2\}$. Let $G$ be the graph $K_{\binom{2d-1}{d},\binom{2d-1}{d}}$. Then, by Proposition~\ref{d}, we have 
\[\ch(G)>d\geq g(2) = g(\chi(G)).\]
The result follows.
\end{proof}

Although there is no upper bound on $\ch$ in terms of $\chi$ for general graphs, it is still reasonable to consider the relationship between $\ch$ and $\chi$ for more restricted families of graphs. In the coming sections, we will discuss some of the most important results and conjectures in this direction: the List Colouring Conjecture (line graphs), Gravier and Maffray's Conjecture (claw-free graphs), the List Total Colouring Conjecture (total graphs), the List Square Colouring Conjecture (squares of graphs), Ohba's Conjecture (graphs of bounded order), Thomassen's Five Colour Theorem (planar graphs), and Alon and Tarsi's Theorem (planar bipartite graphs). 

\section{The Dinitz Problem and the List Colouring Conjecture}

The original paper of Erd\H{o}s et al.~\cite{ERT} on choosability was inspired by a problem of Dinitz on partial Latin squares:
\begin{quote}
``Given an $m\times m$ array of $m$-sets, is it always possible to choose one from each set, keeping the chosen elements distinct in every row, and distinct in every column?''~\cite{ERT}
\end{quote}

The problem of Dinitz can be reformulated as a choosability problem for the line graph of $K_{m,m}$. To see this, suppose that we label the vertices of each side of the bipartition of $K_{m,m}$ with the integers $1,\dots,m$, where one side of the bipartition is known as the \emph{row} vertices and the other is known as the \emph{column} vertices. Then the edges of $K_{m,m}$ correspond to the entries of an $m\times m$ array in a straightforward manner. Moreover, two edges of $K_{m,m}$ share an endpoint if and only if they correspond to entries that are in the same row or column (depending on whether they share a row vertex or a column vertex). Therefore, we can reformulate the Dinitz Problem as follows.

\begin{DP}[\cite{ERT}]
Is it true that $L\left(K_{m,m}\right)$ is $m$-choosable for all $m$?
\end{DP}

It is easily observed that the chromatic number of $L(K_{m,m})$ is precisely $m$.\footnote{Note that a proper $m$-colouring of $L(K_{m,m})$ is equivalent to a partitioning of the edges of $K_{m,m}$ into $m$ disjoint matchings.} Thus, the Dinitz problem asks whether the choice number of $L(K_{m,m})$ coincides with its chromatic number. After 15 years, Galvin~\cite{Galvin} solved the Dinitz Problem in the positive. In fact, his celebrated proof yielded much more; namely, the line graph of every bipartite multigraph\footnote{In a multigraph, a pair of vertices can be joined by more than one edge; that is, the edge set is a multiset.} satisfies $\ch=\chi$.\footnote{An exposition of Galvin's proof can be found in Aigner and Ziegler's \emph{Proofs from The Book}~\cite{TheBook}.}

\begin{GT}[\cite{Galvin}]
If $G$ is a bipartite multigraph, then $\ch(L(G))=\chi(L(G))$. 
\end{GT}

Galvin's Theorem can also be viewed as evidence for a much stronger conjecture: the famous List Colouring Conjecture.

\begin{LCC}
For every multigraph $G$, $\ch(L(G))=\chi(L(G))$. 
\end{LCC}

The exact origin of the List Colouring Conjecture is not completely clear. The first time that it appeared in print was in a paper of Bollob\'{a}s and Harris~\cite{BH} but, as is explained in~\cite{HC,JensenToft}, it had also been posed independently by many different researchers: Albertson and Collins, Gupta, and Vizing, to name a few. It is widely believed to be very challenging, and has since become one of the central problems in graph colouring. 

Several noteworthy cases of the List Colouring Conjecture have been proven. As we have mentioned, Galvin~\cite{Galvin} proved the List Colouring Conjecture for bipartite multigraphs. In~\cite{CompleteGraph}, H{\"a}ggkvist and Janssen proved that $\ch(L(K_n))\leq n$ for all $n$, which verifies the List Colouring Conjecture for complete graphs of odd order. Using an algebraic technique of Alon and Tarsi~\cite{PlanarBip,Null} known as Combinatorial Nullstellensatz,\footnote{When applied to choosability, Combinatorial Nullstellensatz says the following: to show that a graph $G$ is $k$-choosable, it is enough to prove that a particular term of a certain polynomial (which is defined in terms of an orientation of $G$) has a non-zero coefficient.} Ellingham and Goddyn~\cite{Luis} verified the conjecture for a certain class of planar multigraphs: \emph{if $G$ is a $k$-regular planar multigraph and $\chi(L(G))=k$, then $\ch(L(G))=k$}. However, perhaps the strongest evidence that we have for the List Colouring Conjecture is provided by the following result of Kahn~\cite{Kahn}, which shows that it is asymptotically correct.

\begin{KT}[\cite{Kahn}]
For every multigraph $G$, $\ch(L(G))=(1+o(1))\chi(L(G))$.
\end{KT}

Kahn's proof has been described as a `tour de force' for incorporating many of the central techniques from the probabilistic method (see~\cite{MolloyReed}). His asymptotics have since been improved by Molloy and Reed~\cite{MolloyReed, Near-Opt} using a similar approach. Sanders and Steurer~\cite{MR2419118} showed that, in fact, there is a deterministic algorithm to properly colour a line graph $L(G)$ from lists of size $(1+o(1))\chi(L(G))$ in polynomial time.\footnote{Because of the probabilistic nature of Kahn's proof, it only ensures the existence of such a colouring, but may not give an efficient method of constructing it.}

\section{Chromatic-Choosable Graphs}

In general, we say that a graph which satisfies $\ch=\chi$ is \emph{chromatic-choosable}~\cite{OhbaOrig}. Apart from the List Colouring Conjecture, there are several intriguing conjectures which claim that graphs of special classes are chromatic-choosable. One such conjecture was proposed by Gravier and Maffray in~\cite{ClawFree}. 

\begin{GMC}[\cite{ClawFree}]
Every claw-free graph is chromatic-choosable.
\end{GMC}

It is easily observed that the line graph of a multigraph is claw-free,\footnote{This follows from the simple fact that every edge has only two endpoints.} and therefore Gravier and Maffray's Conjecture would imply the List Colouring Conjecture. In~\cite{ClawFree}, Gravier and Maffray verified their conjecture for a special class of claw-free graphs known as $3$-colourable \emph{elementary} graphs.\footnote{A graph is \emph{elementary} if its edges can be coloured with two colours such that any induced path on three vertices contains edges of both colours.} 

In~\cite{Total}, Borodin, Kostochka and Woodall proposed a version of the List Colouring Conjecture for total graphs. 

\begin{LTCC}[\cite{Total}]
For every multigraph $G$, $\ch(T(G))=\chi(T(G))$. 
\end{LTCC}

This conjecture seems to fit well with other problems in the area. It is widely believed that the colouring properties of total graphs are quite similar to those of line graphs. This is perhaps best illustrated by the relation between  a famous result of Vizing~\cite{VizThm} and the well studied Total Colouring Conjecture, made independently by  Behzad~\cite{Behzad} and Vizing~\cite{VizTotal}. For every graph $G$, Vizing's Theorem says that $\Delta(G)\leq \chi(L(G))\leq \Delta(G)+1$ and the Total Colouring Conjecture claims that $\Delta(G)+1\leq\chi(T(G))\leq \Delta(G)+2$.\footnote{Note that, in both cases, the lower bound is trivial.} Strong evidence for the Total Colouring Conjecture is provided by a result of Molloy and Reed~\cite{TCC}, which says that there is an absolute constant $C$ such that $\chi(T(G))\leq \Delta(G)+C$ for every graph $G$. 

In~\cite{Square}, Kostochka and Woodall made a conjecture regarding squares of graphs which, as we will see, would imply the List Total Colouring Conjecture.

\begin{LSCC}[\cite{Square}]
For every graph $G$, $\ch\left(G^2\right)=\chi\left(G^2\right)$. 
\end{LSCC}

To verify that the List Square Colouring Conjecture implies the List Total Colouring Conjecture, let $G$ be a multigraph and let $H$ be the graph obtained by \emph{subdividing} every edge of $G$; that is, if $u$ and $v$ are joined in $G$ by an edge $e$, then we replace $e$ by a vertex $w_e$ and two edges $e_1$ and $e_2$ joining from $u$ to $w_e$ and from $w_e$ to $v$, respectively. The List Square Colouring Conjecture would imply that 
\[\ch\left(H^2\right)=\chi\left(H^2\right).\]
Therefore, it suffices to show that $H^2=T(G)$. However, this is not hard to verify. We provide a concrete example in Figure~\ref{totalsquare} and leave the general case as an exercise for the reader. 

\begin{figure}[htbp]
\begin{center}
\begin{tikzpicture}[label distance=3.5mm]
    \tikzstyle{every node}=[draw,circle,fill=black,minimum size=4pt,
                            inner sep=0pt]
    \draw (0,0) node (tri11)  {}
        -- ++(0:1.6cm) node (tri12) {}
        -- ++(240:1.6cm) node (tri13) {}
        -- ++(270:0.8cm) node (center2) [label=left:$e$, draw=none,fill=none]{}
        -- ++(270:0.8cm) node (top1) {}
        	 ++(250:1.0cm) node (left1) [draw=none,fill=none]{}
       	   ++(290:1.0cm) node (bottom1) [label=below:$G$] {}
       	   ++(70:1.0cm) node (right1) [draw=none,fill=none]{};
     \draw (tri11)--(tri13);
     \draw plot [smooth, tension=1] coordinates { (top1) (left1) (bottom1) (right1) (top1)};
     
      \draw (4cm,0) node (tri21)  {}
        -- ++(0:0.8cm) node (tri2a) {}
        -- ++(0:0.8cm) node (tri22) {}
        -- ++(240:0.8cm) node (tri2b) {}
        -- ++(240:0.8cm) node (tri23) {}
        -- ++(120:0.8cm) node (tri2c) {}
        -- ++(300:0.8cm) node (tri23repeat) [draw=none,fill=none]{}
        -- ++(270:0.8cm) node (center2) [label=left:$w_e$] {}
        -- ++(270:0.8cm) node (top2) {}
        	 ++(250:1.0cm) node (left2) {}
       	   ++(290:1.0cm) node (bottom2) [label=below:$H$] {}
       	   ++(70:1.0cm) node (right2) {};
     \draw (tri21)--(tri2c);
     \draw plot [smooth, tension=1] coordinates { (top2) (left2) (bottom2) (right2) (top2)};

       \draw (8cm,0) node (tri31)  [draw=none,fill=none]{}
        -- ++(0:0.8cm) node (tri3a)  [draw=none,fill=none]{}
        	 ++(90:0.21cm) node (tria) [draw=none,fill=none]{}
        	 ++(270:0.21cm) node (tri3arepeat) [draw=none,fill=none]{}
        -- ++(0:0.8cm) node (tri32)  [draw=none,fill=none]{}
        -- ++(240:0.8cm) node (tri3b)  [draw=none,fill=none]{}
           ++(315:0.21cm) node (trib) [draw=none,fill=none]{}
        	 ++(135:0.21cm) node (tri3brepeat) [draw=none,fill=none]{}
        -- ++(240:0.8cm) node (tri33)  [draw=none,fill=none]{}
        -- ++(120:0.8cm) node (tri3c)  [draw=none,fill=none]{}
           ++(225:0.21cm) node (tric) [draw=none,fill=none]{}
        	 ++(45:0.21cm) node (tri3crepeat) [draw=none,fill=none]{}
           ++(300:0.8cm) node (tri33repeat) [draw=none,fill=none]{}
        	 ++(180:0.28cm) node (tri33left) [draw=none,fill=none]{}
        	 ++(0:0.56cm) node (tri33right) [draw=none,fill=none]{}
        	 ++(180:0.28cm) node (tri33repeat2) [draw=none,fill=none]{}
        -- ++(270:0.8cm) node (center3)  [draw=none,fill=none]{}
           ++(180:0.24cm) node (center3left) [draw=none,fill=none]{}
           ++(0:0.24cm) node (center3repeat) [draw=none,fill=none]{}
        -- ++(270:0.8cm) node (top3) {}  [draw=none,fill=none]{}
        	 ++(250:1.0cm) node (left3) {}  [draw=none,fill=none]{}
       	   ++(290:1.0cm) node (bottom3)  [draw=none,fill=none]{}
       	   ++(70:1.0cm) node (right3) {}  [draw=none,fill=none]{}
       	   ++(110:1.0cm) node (top3repeat) [draw=none,fill=none]{}
       	   ++(180:0.27cm) node (top3left) [draw=none,fill=none]{}
       	   ++(0:0.52cm) node (top3right) [draw=none,fill=none]{}
       	   ++(180:0.27cm) node (top3repeat2) [draw=none,fill=none]{}
       	   ++(270:1.05cm) node (middle3) [draw=none,fill=none]{};
     \draw (tri31)--(tri3c);
     \draw (tri31)--(tri3a);
     \draw (tri32)--(tri3a);
     \draw (tri32)--(tri3b);
     \draw (tri33)--(tri3b);
     \draw (tri33)--(tri3c);
     \draw (tri33)--(center3);
     \draw (top3)--(center3);
     
     \draw plot [smooth, tension=1] coordinates { (top3) (left3) (bottom3) (right3) (top3)};     
    \draw (top3)--(bottom3)[color=red];
    \draw [red] plot [smooth, tension=1] coordinates { (tri31) (tria) (tri32)};
    \draw [red] plot [smooth, tension=1] coordinates { (tri32) (trib) (tri33)};
    \draw [red] plot [smooth, tension=1] coordinates { (tri33) (tric) (tri31)};
    \draw [red] plot [smooth, tension=1] coordinates { (tri33) (center3left) (top3)};
    
    \draw (tri3a)--(tri3b) [color=blue];
    \draw (tri3b)--(tri3c)[color=blue];
    \draw (tri3c)--(tri3a)[color=blue];
    \draw [blue] plot [smooth, tension=1] coordinates { (tri3c) (tri33left) (center3)};
    \draw [blue] plot [smooth, tension=1] coordinates { (tri3b) (tri33right) (center3)};
    \draw [blue] plot [smooth, tension=1] coordinates { (left3) (top3left) (center3)};
    \draw [blue] plot [smooth, tension=1] coordinates { (right3) (top3right) (center3)};
  
    \draw [blue] plot [smooth, tension=1] coordinates { (left3) (middle3) (right3)};
    
    
    \draw (8cm,0) node (COPYtri31)  {}
         ++(0:0.8cm) node (COPYtri3a) {}
        	 ++(90:0.21cm) node (COPYtria) [draw=none,fill=none]{}
        	 ++(270:0.21cm) node (COPYtri3arepeat) [draw=none,fill=none]{}
         ++(0:0.8cm) node (COPYtri32) {}
         ++(240:0.8cm) node (COPYtri3b) {}
           ++(315:0.21cm) node (COPYtrib) [draw=none,fill=none]{}
        	 ++(135:0.21cm) node (COPYtri3brepeat) [draw=none,fill=none]{}
         ++(240:0.8cm) node (COPYtri33) {}
         ++(120:0.8cm) node (COPYtri3c) {}
           ++(225:0.21cm) node (COPYtric) [draw=none,fill=none]{}
        	 ++(45:0.21cm) node (COPYtri3crepeat) [draw=none,fill=none]{}
           ++(300:0.8cm) node (COPYtri33repeat) [draw=none,fill=none]{}
        	 ++(180:0.28cm) node (tri33left) [draw=none,fill=none]{}
        	 ++(0:0.56cm) node (COPYtri33right) [draw=none,fill=none]{}
        	 ++(180:0.28cm) node (COPYtri33repeat2) [draw=none,fill=none]{}
         ++(270:0.8cm) node (COPYcenter3) [label=left:$w_e$] {}
           ++(180:0.24cm) node (COPYcenter3left) [draw=none,fill=none]{}
           ++(0:0.24cm) node (COPYcenter3repeat) [draw=none,fill=none]{}
         ++(270:0.8cm) node (COPYtop3) {}
        	 ++(250:1.0cm) node (COPYleft3) {}
       	   ++(290:1.0cm) node (COPYbottom3) [label=below:$H^2$] {}
       	   ++(70:1.0cm) node (COPYright3) {};
       	            
\end{tikzpicture}
\end{center}
\caption{\label{totalsquare}Constructing the total graph of $G$ by subdividing its edges and then taking the square. In $H^2$, the red edges correspond to adjacencies in $G$ and the blue edges correspond to the edges in $L(G)$.}
\end{figure}
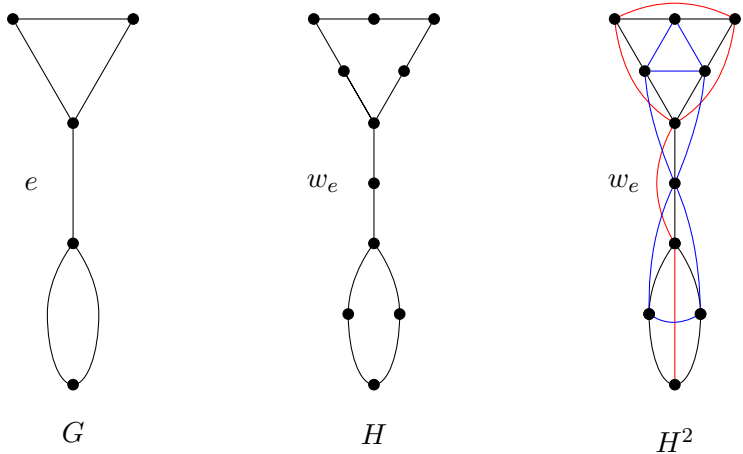 

Thus, the List Square Colouring Conjecture would imply the List Total Colouring Conjecture. Furthermore, one should observe that the graph $H$ obtained by subdividing every edge of $G$ is a bipartite graph in which one side of the bipartition consists of vertices of degree $2$. Therefore, the List Total Colouring Conjecture corresponds exactly to this special case of the List Square Colouring Conjecture. 

In this thesis, we consider problems of a somewhat different flavour. A key feature of graphs for which $\ch$ is large with respect to $\chi$ is that they must contain a large number of vertices (c.f. Corollary~\ref{nofunc}). Thus, the following question is quite natural: 
\begin{ques}
\label{h}
Given a function $h$ on $\mathbb{N}$, what is the best upper bound on $\ch$ in terms of $\chi$ for graphs $G$ on at most $h(\chi(G))$ vertices?
\end{ques}
Our main focus is a conjecture of this type due to Ohba~\cite{OhbaOrig}.\footnote{Bondy and Murty included Ohba's Conjecture as one of the unsolved problems in their recent book \emph{Graph Theory}~\cite{BondyMurtyNew}.}

\begin{OC}[\cite{OhbaOrig}]
If $|V(G)|\leq 2\chi(G)+1$, then $\ch(G)=\chi(G)$.
\end{OC}

Of course, if the hypothesis of Ohba's Conjecture is replaced with $|V(G)|\leq \chi(G)$, then we obtain a trivial statement regarding the choice number of a complete graph. In this sense, Ohba's Conjecture says that if a graph $G$ is \emph{close enough} to being a complete graph, then it must be chromatic-choosable. Note that the example of $K_{3,3}$ from Figure~\ref{K3,3} satisfies $|V|=2\chi+2$ and $\ch>\chi$ and so Ohba's Conjecture is best possible. 

A simple observation is that the operation of adding an edge between vertices of $G$ in different colour classes of a $\chi(G)$-colouring does not change the chromatic number and does not decrease the choice number. By adding all such edges, we see that Ohba's Conjecture is true if and only if it is true for complete multipartite graphs. Therefore, we can restate Ohba's Conjecture as follows.

\begin{OC}[\cite{OhbaOrig}]
If $G$ is a complete $k$-partite graph on at most $2k+1$ vertices, then $\ch(G)=k$.
\end{OC}

We will discuss the history of Ohba's Conjecture in more detail, including motivating examples and partial results, in the next chapter. In Chapter~\ref{proof}, we present a full proof of Ohba's Conjecture from~\cite{NRW}.  

\section{Other Notable Results}

Returning briefly to the origins of graph colouring, it is natural to wonder if there is a `choosability analog' of the Four Colour Theorem. This question dates back to the original papers on choosability; Vizing~\cite{Vizing} asked whether every planar graph is $5$-choosable, and Erd\H{o}s et al.~\cite{ERT} independently conjectured that every planar graph is $5$-choosable but not necessarily $4$-choosable. Voigt~\cite{Voigt} discovered the first example of a planar graph which is not $4$-choosable and, later, an example of smaller order was discovered by Mirzakhani~\cite{Mirzakhani}. In~\cite{Thomassen}, Thomassen gave a simple and beautiful proof that every planar graph is $5$-choosable.\footnote{An exposition of Thomassen's proof can be found in Aigner and Ziegler's \emph{Proofs from The Book}~\cite{TheBook}.}

\begin{T5CT}[\cite{Thomassen}]
Every planar graph is $5$-choosable.
\end{T5CT} 

Thomassen's result has since been generalized to many different classes of graphs which are, in one way or another, \emph{close} to being planar; some notable examples are as follows: $K_5$-minor free graphs~\cite{K5MinorFree} (for alternate proofs, see~\cite{K5MinorFree2,K5MinorFree3}), locally planar graphs~\cite{LocallyPlanar}, graphs with crossing number at most $2$~\cite{TwoCrossings} (independently proved in~\cite{Victor}), and graphs $G$ such that $G-e$ is planar for some edge $e$~\cite{Victor}. Moreover, Thomassen improved his own result by showing that for any assignment $L$ of lists of size $5$ to the vertices of a planar graph, there are exponentially many acceptable colourings for $L$~\cite{ExpoMany}.  

Another problem of Erd\H{o}s et al.~\cite{ERT} was to determine whether every planar bipartite graph is $3$-choosable. In~\cite{PlanarBip}, Alon and Tarsi solved this problem using Combinatorial Nullstellensatz. Their result is best possible since $K_{4,2}$ is planar and $\ch(K_{4,2})>2$, as one can verify by considering Figure~\ref{K4,2}.

\begin{ATT}[\cite{PlanarBip}]
Every planar bipartite graph is $3$-choosable.
\end{ATT} 

Note that, since $\left|V\left(K_{4,2}\right)\right|=2\chi\left(K_{4,2}\right)+2$, we see that $K_{4,2}$ is another example which shows that Ohba's Conjecture is best possible with respect to the bound on $|V|$. 

\begin{figure}[htbp]
\begin{center}
\begin{tikzpicture}[label distance=1mm]
    \tikzstyle{every node}=[draw,circle,fill=black,minimum size=4pt,
                            inner sep=0pt]
    \draw (0,2cm) node (13) [label={[label distance=-1.5mm]90:$\{{1,3}\}$}] {}
    			++(270:1.32cm) node (14) [label={[label distance=-1.5mm]90:$\{{1,4}\}$}] {}
    			++(270:1.32cm) node (23) [label={[label distance=-1.5mm]270:$\{{2,3}\}$}] {}
    			++(270:1.32cm) node (24) [label={[label distance=-1.5mm]270:$\{{4,2}\}$}] {}
        	++(45:2.8001428534987281966273436739336cm) node (12R) [label=right:$\{{1,2}\}$] {}
        	++(180:3.96cm) node (34L) [label=left:$\{{3,4}\}$] {};
        	 
     \draw (13)--(34L);
     \draw (13)--(12R);
     \draw (14)--(34L);
     \draw (14)--(12R);
     \draw (23)--(34L);
     \draw (23)--(12R);
     \draw (24)--(34L);
     \draw (24)--(12R);
\end{tikzpicture}
\end{center}
\caption{\label{K4,2}A drawing of $K_{4,2}$ showing that it is planar, and a list assignment showing that it is not $2$-choosable.}
\end{figure}
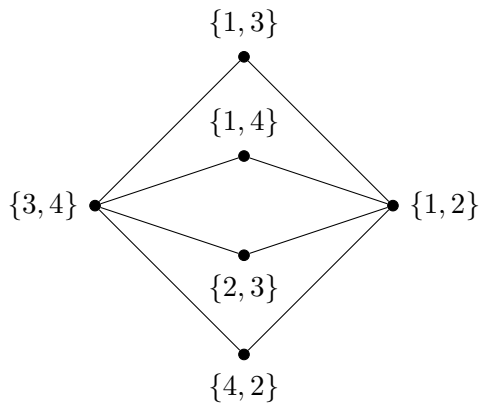

\let\newpage\relax
\let\clearpage\relax


}


{

{
\if@openright
    \cleardoublepage
  \else
    \clearpage
  \fi
\let\newpage\relax

\part{Ohba's Conjecture}
\label{OhbaPart}

}
\let\newpage\relax
\let\clearpage\relax

{\addcontentsline{toc}{part}{Part~\ref{OhbaPart}: Ohba's Conjecture}
\newpage
\clearpage

\chapter{Background}
\label{background}

\begin{chapquote}{David Hilbert}
The art of doing mathematics consists in finding that special case which contains all the germs of generality.\end{chapquote}
\doublespacing

Ohba's Conjecture is motivated by some simple, yet illuminating, examples. The first of these comes from the original paper of Erd\H{o}s et al.~\cite{ERT} on choosability. By applying Hall's Theorem on matchings in bipartite graphs, they proved that the complete $k$-partite graph in which every part has size $2$ is $k$-choosable. The core ideas of their argument, which we present now, will resurface in the proof of Ohba's Conjecture in the next chapter. 

\begin{thm}[Erd\H{o}s et al.~\cite{ERT}]
\label{222}
If $G$ is a complete multipartite graph in which every part has size $2$, then $\ch(G)=k$.
\end{thm}

A key idea in the proof of Theorem~\ref{222} is that it is sometimes useful to view a choosability problem in terms of a matching problem on a special bipartite graph. This idea is captured by the following definition. 

\begin{defn}
\label{bipL}
Given a graph $G$ and list assignment $L$, let 
\begin{enumerate}[leftmargin=*,labelindent=1em,label=$\bullet$]
\item $C_L:=\cup_{v\in V(G)}L(v)$,
\item $B_L$ be the bipartite graph with bipartition $\left(V(G),C_L\right)$ where each $v\in V(G)$ is joined to the colours of $L(v)$.
\end{enumerate}
\end{defn}

Clearly if there is a matching $M$ in $B_L$ which saturates $V(G)$, then an acceptable colouring for $L$ can be obtained by simply mapping each vertex of $G$ to the colour that it is matched to under $M$. To prove Theorem~\ref{222} we show that we can always either find a matching in $B_L$ which saturates $V(G)$, or reduce our problem to a smaller one and apply induction on $k$. In order to do so, we apply a well known theorem of Hall~\cite{Hall} from matching theory.\footnote{Hall's Theorem is often stated in terms of a \emph{system of distinct representatives} for a collection of sets; see Chapter~\ref{strength} for this formulation.}

\begin{HT}[\cite{Hall}]
Let $B$ be a bipartite graph with bipartition $(X,Y)$ and let $S\subseteq X$. Then there is a matching $M$ which saturates $S$ if and only if $\left|N_B(T)\right|\geq |T|$ for every subset $T$ of $S$.
\end{HT}

We are now in position to prove Theorem~\ref{222}.

\begin{proof}[Proof of Theorem~\ref{222}]
Suppose that Theorem~\ref{222} is false for some $k$ and let $G$ be  complete $k$-partite graph in which every part has size $2$. Moreover, suppose that $k$ is the smallest such integer and let $L$ be an assignment of lists of size $k$ to the vertices of $G$ such that there does not exist an acceptable colouring for $L$. 

\begin{lem}
\label{emptycap}
If $P=\{u,v\}$ is a part of $G$, then $L(u)\cap L(v)=\emptyset$.
\end{lem}

\begin{proof}
Otherwise, let $c\in L(u)\cap L(v)$. We delete $\{u,v\}$ from $G$ and remove $c$ from the lists of the remaining vertices. Clearly the remaining graph is a complete $(k-1)$-partite graph in which every part has size $2$, and every remaining vertex has at least $k-1$ colours remaining in its list. Therefore, by our choice of $k$, we can find an acceptable colouring of $G-\{u,v\}$ from the remaining lists and extend it to an acceptable colouring for $L$ by mapping both $u$ and $v$ to $c$. This contradicts our choice of $L$.  
\end{proof}

The following lemma completes the proof of Theorem~\ref{222}.

\begin{lem}
\label{match2}
There is a matching in $B_L$ which saturates $V(G)$.
\end{lem}

\begin{proof}
If not, then by Hall's Theorem there is a set $S\subseteq V(G)$ such that $\left|N_{B_L}(S)\right|<|S|$. If $S$ contains both vertices from some part $P=\{u,v\}$ of $G$, then by Lemma~\ref{emptycap} we must have
\[\left|N_{B_L}(S)\right|\geq |L(u)\cup L(v)| = |L(u)|+|L(v)|\geq 2k=|V(G)|\geq |S|\]
contradicting our choice of $S$. Thus, $S$ must contain at most one vertex from each part of $G$, and so it contains at most $k$ vertices in total. However, for any vertex $v\in S$ we have
\[\left|N_{B_L}(S)\right|\geq |L(v)|\geq k\geq |S|\]
which, once again, contradicts our choice of $S$. Therefore, no such set $S$ can exist. This completes the proof of Lemma~\ref{match2} and of Theorem~\ref{222}.
\end{proof}
\end{proof}

Later, Gravier and Maffray~\cite{GravMaf} extended Theorem~\ref{222} to the case where one of the parts has size $3$ and the rest have size $2$. Note that Gravier and Maffray's result concerns graphs of order exactly $2\chi+1$; that is, these graphs satisfy the hypothesis of Ohba's Conjecture with equality. 

\begin{thm}[Gravier and Maffray~\cite{GravMaf}]
\label{3222}
If $G$ is a complete $k$-partite graph with $1$ part of size $3$ and $k-1$ parts of size $2$, then $\ch(G)=k$.
\end{thm}

As we have already seen, the graphs $K_{3,3}$ and $K_{4,2}$ can be used to show that Ohba's Conjecture is best possible. Enomoto, Ohba, Ota and Sakamoto~\cite{Examples} expanded on these examples to obtain two infinite classes of graphs which satisfy $|V|=2\chi+2$ and $\ch>\chi$. 

\begin{prop}[Enomoto et al.~\cite{Examples}]
\label{4222}
Let $k\geq 2$ be even. If $G$ is the complete $k$-partite graph with $1$ part of size $4$ and $k-1$ parts of size $2$, then $\ch(G)> k$. 
\end{prop}

\begin{proof}
Let $A_1,A_2,B_1,B_2$ be disjoint sets of $\frac{k}{2}$ colours and define $A:=A_1\cup A_2$ and $B:=B_1\cup B_2$. Let $L$ be the list assignment such that every part of size $2$ contains one vertex assigned to each of the lists $A$ and $B$, and every part of size $4$ contains one vertex assigned to each of the lists $A_1\cup B_1, A_1\cup B_2, A_2\cup B_1$ and $A_2\cup B_2$. 

Now, suppose that there is an acceptable colouring $f$ for $L$. Then, since $f$ is acceptable, it must map the parts of size $2$ to a set of exactly $k-1$ colours from each of $A$ and $B$. This implies that for some $i,j\in\{1,2\}$ every colour of $A_i\cup B_j$ is used by $f$ on the parts of size $2$. Therefore, it is impossible to colour the vertex of the part of size $4$ whose list is $A_i\cup B_j$, and so there can be no acceptable colouring for $L$.
\end{proof}

\begin{prop}[Enomoto et al.~\cite{Examples}]
\label{3331}
Let $s\geq1$ and define $k=2s$. If $G$ is the complete $k$-partite graph with $s+1$ parts of size $3$ and $s-1$ parts of size $1$, then $\ch(G)> k$. 
\end{prop}

\begin{proof}
Let $X,Y,Z$ be three disjoint sets of exactly $s$ colours and define $C:=X\cup Y\cup Z$. Let $L$ be the list assignment such that the vertex of each part of size $1$ is assigned to the list $C$ and every part of size $3$ contains one vertex assigned to each of the lists $X\cup Y$, $X\cup Z$ and $Y\cup Z$.\footnote{Note that each vertex in a singleton part is assigned to a list of $3s>k$ available colours.}

Now, suppose that there is an acceptable colouring $f$ for $L$. Then $f$ must map each part of size $3$ to at least $2$ colours, and each part of size $1$ to exactly $1$ colour, where all of these colours are distinct. Thus, in total, $f(V(G))$ contains at least
\[2(s+1)+(s-1) = 3s+1>|C|\]
distinct colours. However, this contradicts the basic assumption that $f(V(G))\subseteq C$. The result follows. 
\end{proof}

The above examples demonstrate the importance of considering the distribution of colours in the lists. In Proposition~\ref{4222}, the reason that there does not exist an acceptable colouring for $L$ is that the lists assigned to each part are too \emph{spread apart}. That is, we are never given the opportunity to colour the vertices in a part of size $2$ with the same colour, and this forces us to deplete the lists of the vertices in the part of size $4$. On the other hand, in Proposition~\ref{3331}, the issue is that the \emph{total number} of colours is too small, which suggests that the lists assigned are too \emph{close together}. In proving our main results, we will need to pay special attention to each of these two extremes.

The only known examples of complete $k$-partite graphs on $2k+2$ vertices for which $\ch>k$ are covered by Propositions~\ref{4222} and~\ref{3331}. We conjecture that these are the only such examples. 

\begin{conj}
\label{only}
If $G$ is a complete $k$-partite graph on $2k+2$ vertices such that $\ch(G)>k$, then $G$ satisfies the hypothesis of Proposition~\ref{4222} or~\ref{3331}.
\end{conj}

Regarding Conjecture~\ref{only}, we remark that Gravier and Maffray proved that for $k\geq 3$ the complete $k$-partite graph with $2$ parts of size $3$ and $k-2$ parts of size $2$ is chromatic-choosable. Also, Enomoto et al.~\cite{Examples} proved that for odd $k$, the complete $k$-partite graph with $1$ part of size $4$ and $k-1$ parts of size $2$ is chromatic-choosable. A particular consequence of Conjecture~\ref{only}, which is alluded to in~\cite{Examples}, would be that Ohba's Conjecture could be extended to graphs of order $2\chi+2$ which have odd chromatic number.

\section{Partial Results}

Ohba's Conjecture has drawn considerable attention, and many notable cases have been proven. One popular approach has been to prove variants of Ohba's Conjecture which assume a stronger bound on $|V|$. In his original paper on the conjecture, Ohba~\cite{OhbaOrig} proved a result of this type. 

\begin{thm}[Ohba~\cite{OhbaOrig}]
\label{OhbaThm}
If $|V(G)|\leq \chi(G)+\sqrt{2\chi(G)}$, then $\ch(G)=\chi(G)$. 
\end{thm}

Note that the above theorem does not provide us with a real number $a>1$ for which all graphs on at most $a\chi$ vertices are chromatic-choosable. The first result of this type was proved by Reed and Sudakov~\cite{ReedSudakov} using a probabilistic approach.

\begin{thm}[Reed and Sudakov~\cite{ReedSudakov}]
If $|V(G)|\leq \frac{5}{3}\chi(G)-\frac{4}{3}$, then $\ch(G)=\chi(G)$.
\end{thm}

Using more sophisticated probabilistic tools, Reed and Sudakov~\cite{ReedSudakov2} also obtained an asymptotic version of Ohba's Conjecture. Their result provides us with strong evidence that Ohba's Conjecture is, at the very least, nearly correct.

\begin{thm}[Reed and Sudakov~\cite{ReedSudakov2}]
If $|V(G)|\leq (2-o(1))\chi(G)$, then $\ch(G)=\chi(G)$. 
\end{thm}

Another approach to Ohba's Conjecture has been to verify it for graphs of bounded stability number. An early result in this direction is due to Ohba~\cite{3Ohba}, who proved that if $|V(G)|\leq 2\chi(G)$ and $\alpha(G)\leq 3$, then $G$ is chromatic-choosable. By building on Ohba's techniques, He, Li, Shen and Zheng~\cite{Alpha=3} extended this result to graphs of order $2\chi+1$.

\begin{thm}[He et al.~\cite{Alpha=3}]
If $|V(G)|\leq 2\chi(G)+1$ and $\alpha(G)\leq 3$, then $\ch(G)=\chi(G)$. 
\end{thm}

Later, Kostochka, Stiebitz and Woodall~\cite{Kostochka} proved Ohba's Conjecture for graphs of stability number at most $5$. Recall that the only complete multipartite graphs known to satisfy $|V|=2\chi+2$ and $\ch>\chi$ have stability number at most $4$ (c.f. Conjecture~\ref{only}). In light of this, their result seems to cover most of the `potential counterexamples' to Ohba's Conjecture. 

\begin{thm}[Kostochka et al.~\cite{Kostochka}]
\label{alpha5}
If $|V(G)|\leq 2\chi(G)+1$ and $\alpha(G)\leq 5$, then $\ch(G)=\chi(G)$. 
\end{thm}

Their technique can be roughly outlined as follows. Choose a stable set $S\subseteq V(G)$ and replace it with a single vertex whose list is $\cap_{v\in S}L(v)$ and whose neighbourhood is $\cup_{v\in S}N(v)$, and repeat this procedure until some stopping condition is reached. Then apply Hall's Theorem (in a similar fashion to the proof of Lemma~\ref{match2}) to obtain an acceptable colouring for the resulting choosability problem. We will apply a variant of this approach to prove a strengthening of Ohba's Conjecture in Chapter~\ref{strength}.

Additionally, some papers have focused on proving Ohba's Conjecture for specific classes of complete multipartite graphs~\cite{MR2370527,MR2319821,MR2459412,MR2683044}. Typically, these results are proved by combining Hall's Theorem with detailed case analysis. 


}
}

{
\newpage
\clearpage
\chapter{The Proof}
\label{proof}

\begin{chapquote}{Paul R. Halmos, \emph{I Want to Be a Mathematician}}
Mathematics is not a deductive science -- that's a clich\'{e}. When you try to prove a theorem, you don't just list the hypotheses, and then start to reason. What you do is trial and error, experimentation, guess-work. \end{chapquote}

\doublespacing 

In order to prove Ohba's Conjecture, we will develop several tools involving Hall's Theorem and the special bipartite graph $B_L$ introduced in Definition~\ref{bipL}. The first of these, which we prove in the next section, is a lemma pertaining to general graphs that was discovered independently by Kierstead~\cite{Kierstead} and Reed and Sudakov~\cite{ReedSudakov, ReedSudakov2} (in slightly different forms). We refer to it as the Colour Matching Lemma.

\begin{CML}[Kierstead~\cite{Kierstead}; Reed and Sudakov~\cite{ReedSudakov, ReedSudakov2}]
\label{match}
Suppose that $G$ is not $k$-choosable, and let $L$ be a list assignment such that
\begin{enumerate}[leftmargin=*,labelindent=1em,label=$\bullet$]
\item $|L(v)|\geq k$ for all $v\in V(G)$,
\item there is no acceptable colouring for $L$, and
\item subject to this, $|C_L|$ is minimum.
\end{enumerate}
Then there is a matching in $B_L$ which saturates $C_L$.
\end{CML}

As we have mentioned, if there is a matching $M$ in $B_L$ which saturates $V(G)$, then we can obtain an acceptable colouring for $L$ by mapping each vertex of $G$ to the colour that it is matched to under $M$ (c.f. Lemma~\ref{match2}). So, if there is a matching in $B_L$ which saturates $C_L$, but there is no acceptable colouring for $L$, then it must be the case that  $|C_L|<|V(G)|$. Therefore, we make the following observation.

\begin{obs}
\label{<n}
If $L$ is a list assignment as in the Colour Matching Lemma, then $\left|C_L\right|<|V(G)|$. Therefore, in order to determine the choice number of a graph $G$, we can restrict our attention to list assignments $L$ which satisfy $\left|C_L\right|<|V(G)|$.
\end{obs}

The Colour Matching Lemma seems to be widely applicable in choosability problems, especially for graphs of bounded order. In Chapter~\ref{strength}, we will apply Observation~\ref{<n} to prove a strengthening of Ohba's Conjecture. 

Our proof of Ohba's Conjecture is by contradiction. Suppose that Ohba's Conjecture is false and let $G$ and $L$ be chosen to have the following properties:
\begin{enumerate}[leftmargin=*,labelindent=1em,label=$\bullet$]
\item $G$ is a complete $k$-partite graph on at most $2k+1$ vertices, and
\item $L$ is a list assignment of $G$ as in the Colour Matching Lemma.
\end{enumerate}
Furthermore, we can choose $G$ to be a \emph{minimal counterexample} in the sense that Ohba's Conjecture is true for all graphs on fewer than $|V(G)|$ vertices. By Observation~\ref{<n} and our choice of $L$, the difference $|V(G)|-\left|C_L\right|$ is positive; it will be useful for us to keep track of this quantity. 

\begin{defn}
$\gamma:=|V(G)|-\left|C_L\right|$.
\end{defn}

To prove the conjecture, our general approach is to use the minimality of $G$ to obtain information about the distribution of colours in the lists, and to use this information to construct an acceptable colouring for $L$. Of particular importance are the lists of vertices which form singleton parts of $G$. 

\begin{defn} 
Say that a vertex $v\in V(G)$ is a \emph{singleton} if $\{v\}$ is a part of $G$. 
\end{defn}

The proof of Ohba's Conjecture is divided into three lemmas, which we refer to as Lemmas~\ref{na},~\ref{colproc} and~\ref{enough}. We state these lemmas here and prove them later in the chapter. In order to state them, we require two more definitions. 

\begin{defn}
We say that a colour $c\in C_L$ is
\begin{enumerate}[leftmargin=*,labelindent=1em,label=$\bullet$]
\item \emph{globally frequent} if $c$ is available for at least $k+1$ vertices of $G$,
\item \emph{frequent among singletons} if $c$ is available for at least $\gamma$ singletons.
\end{enumerate}
If $c$ is either globally frequent or frequent among singletons, then we say that $c$ is \emph{frequent}.
\end{defn}

\begin{defn}
We say that a proper colouring $f:V(G)\to C_L$ is \emph{near-acceptable} for $L$ if for every vertex $v\in V(G)$, either
\begin{enumerate}[leftmargin=*,labelindent=1em,label=$\bullet$]
\item $f(v)\in L(v)$, or
\item $f^{-1}(f(v))=\{v\}$ and $f(v)$ is frequent.
\end{enumerate}
\end{defn}

Of course, every acceptable colouring for $L$ is also near-acceptable for $L$, but the converse is not true. However, a near-acceptable colouring for $L$ can be modified to obtain an acceptable colouring for $L$, as the following lemma suggests. 

\begin{lemma}
\label{na}
If there is a near-acceptable colouring for $L$, then there is an acceptable colouring for $L$.
\end{lemma}

The advantage of Lemma~\ref{na} is that a near-acceptable colouring for $L$ is sometimes easier to construct than an acceptable colouring for $L$. Specifically, when constructing a near-acceptable colouring for $L$, a vertex $v$ of $G$ can be mapped to a colour $c$ outside of $L(v)$ provided that $c$ is frequent and no other vertex is mapped to $c$. We will exploit this flexibility to prove the following.

\begin{lemma}
\label{colproc}
If $C_L$ contains at least $k$ frequent colours, then there is a near-acceptable colouring for $L$.
\end{lemma}

Finally we will show that, if the Ohba's Conjecture is false, then there is a counterexample which satisfies the hypothesis of Lemma~\ref{colproc}, thereby completing the proof.

\begin{lemma}
\label{enough}
There is a list assignment $L$ as in the Colour Matching Lemma such that $C_L$ contains at least $k$ frequent colours. 
\end{lemma}

The rest of the chapter is outlined as follows. In the next section, we prove the Colour Matching Lemma for general graphs, and discuss some of its consequences regarding the proof of Ohba's Conjecture. In Section~\ref{mce}, we use the minimality assumption on $G$ to obtain some basic properties of $G$ and $L$ which we exploit repeatedly in the rest of the proof. 

Then, we begin to prove the main lemmas, starting with a proof of Lemma~\ref{na} in Section~\ref{outside}. The basic ideas of the proof are as follows. First, if $f$ is a near-acceptable colouring for $L$ and $v^*$ is mapped by $f$ to a colour $c^*\notin L(v^*)$, then we know that $f^{-1}(c^*)=\{v^*\}$ and that $c^*$ is either globally frequent or frequent among singletons. We proceed differently based on these two cases. 

In the first case, by uncolouring $v^*$, we obtain a colouring of $G-v^*$ which does not map to $c^*$. Now, since $c^*$ is globally frequent, there are at least $k+1$ vertices of $G$ for which $c^*$ is available, and at least $k$ colours of $C_L$ which are available for $v^*$. Using this information, we apply Hall's Theorem to show that we can modify $f$ so that $v^*$ is mapped to a colour of $L(v^*)$ and $c^*$ is used on a stable set for which it is available.

The second case requires a more detailed argument; we provide only a few of the ideas here. Since $c^*$ is frequent among singletons, there exists a set of at least $\gamma$ singletons for which $c^*$ is available. If we delete a set $A\subseteq V(G)$ from $G$ such that $A$ contains all such singletons and at least $\gamma$ additional vertices, then the chromatic number of $G$ decreases by at least $\gamma$ and the order of $G$ decreases by at least $2\gamma$. Therefore, by minimality of $G$, we see that $G-A$ is $(k-\gamma)$-choosable. Given a careful choice of $A$, we show that there is an acceptable colouring of $G-A$ from lists of size $k-\gamma$ which we can combine with a colouring of $G[A]$, constructed via Hall's Theorem, to obtain an acceptable colouring for $L$. 

In Section~\ref{procSec}, we prove Lemma~\ref{colproc} by describing an explicit greedy procedure which allows us to construct a near-acceptable colouring for $L$ provided that $C_L$ contains a set $F$ of $k$ frequent colours. The procedure consists of three phases. Loosely speaking, the goal of the first two phases is to use all of the colours of $C_L-F$ and a few colours of $F$ to properly (and greedily) colour as many vertices as we can, where every vertex coloured in the first two phases is mapped to a colour in its list. The goal is to show that, after the first two phases, the number of unused colours of $F$ is at least the number of vertices of $G$ which have not yet been coloured. If this is the case, then we obtain a near-acceptable colouring for $L$ by assigning each unused colour of $F$ to at most one vertex of $G$ that was not coloured in the first two phases, where this vertex is chosen arbitrarily. 

Finally, in Sections~\ref{commontosing} and~\ref{count}, we consider a certain type of extremal list assignment $L$ and analyze its properties to prove Lemma~\ref{enough}. One particularly useful property of such an extremal list assignment $L$, which we will derive in Section~\ref{commontosing}, is that a colour $c\in C_L$ is frequent if and only if it is available for every singleton of $G$. In Section~\ref{count}, we will apply a strengthening of Lemma~\ref{na} and many careful counting arguments to complete the proof of Lemma~\ref{enough}, and of Ohba's Conjecture.

\section{The Colour Matching Lemma}
\label{matchSec}

In this section, we prove the Colour Matching Lemma.

\begin{proof}[Proof of the Colour Matching Lemma]
Let $G$ be a graph which is not $k$-choosable and let $L$ satisfy the hypotheses of the Colour Matching Lemma. We show that there is a matching in $B_L$ which saturates $C_L$. 

Otherwise, by Hall's Theorem, there must be a set $S\subseteq C_L$ such that $\left|N_{B_L}(S)\right|<|S|$. Let $S$ be such a set chosen so that $|S|$ is minimum. We choose $c\in S$ arbitrarily and define $T:=S-c$. By our choice of $S$ we see that, for every subset $T'$ of $T$,
\[\left|N_{B_L}(T')\right|\geq |T'|.\]
Therefore, by Hall's Theorem, there is a matching $M$ in $B_L$ which saturates $T$. Moreover, by definition of $S$ and $T$ we have
\[\left|N_{B_L}(S)\right|\geq \left|N_{B_L}(T)\right|\geq \left|T\right|=|S|-1\geq \left|N_{B_L}(S)\right|\]
and so equality must hold throughout. In particular, $|T|=\left|N_{B_L}(T)\right|$ and so $M$ must also saturate $N_{B_L}(T)$. For each vertex $v\in N_{B_L}(T)$, let $g(v)$ be the colour in $T$ which is matched to $v$ by $M$. Since $M$ saturates $N_{B_L}(T)$, we have that $N_{B_L}(T)$ cannot contain all of $V(G)$ for, if it did, then $g$ would be an acceptable colouring for $L$. Thus, we can choose $w\in V(G)-N_{B_L}(T)$ arbitrarily. 

Now, define a list assignment $L'$ of $G$ in the following way:
\[L'(v):=\left\{\begin{array}{ll}	L(w) & \text{if } v\in N_{B_L}(T),\\
																	L(v) & \text{otherwise}.\end{array}\right.\]
Clearly $|L'(v)|\geq k$ for all $v\in V(G)$ and 
\begin{equation}\label{noT}C_{L'}\subseteq C_L-T\neq C_L.\end{equation}
Therefore, since $L$ satisfies the hypotheses of the Colour Matching Lemma, there must exist an acceptable colouring $f'$ for $L'$. Moreover, by (\ref{noT}), we see that $f'$ does not map to any colour of $T$. Thus, we can construct an acceptable colouring $f$ for $L$ as follows 
\[f(v):=\left\{\begin{array}{ll}	g(v) & \text{if } v\in N_{B_L}(T),\\
																	f'(v) & \text{otherwise}.\end{array}\right.\]
This contradicts our choice of $L$ and completes the proof. 
\end{proof}

A matching in $B_L$ which saturates $C_L$ corresponds to an injective function from $C_L$ to $V(G)$ in a straightforward way. The next proposition captures this.

\begin{prop}
\label{inj}
If $L$ is a list assignment as in the Colour Matching Lemma, then there is an injective function $h:C_L\to V(G)$ such that $c\in L(h(c))$ for every $c\in C_L$. 
\end{prop}

\begin{proof}
Let $M$ be a matching in $B_L$ which saturates $C_L$, and for each $c\in C_L$ let $h(c)$ be the vertex that it is matched to under $M$. 
\end{proof}

Returning to the proof of Ohba's Conjecture, one consequence of the Colour Matching Lemma is that every near-acceptable colouring for $L$ can be modified to obtain a near-acceptable colouring for $L$ which maps surjectively to $C_L$. This is implied by the next proposition.

\begin{prop}
\label{surj}
If $f:V(G)\to C_L$ is a proper colouring, then there is a proper surjective colouring $g:V(G)\to C_L$ such that for each $v\in V(G)$, either
\begin{enumerate}[leftmargin=*,labelindent=1em,label=\rm{(\alph*)}]
\item \label{ggood} $g(v)\in L(v)$, or
\item \label{gcontained} $g^{-1}(g(v))\subseteq f^{-1}(f(v))$. 
\end{enumerate}
\end{prop}

\begin{proof}
Let $h:C_L\to V(G)$ be a function as in Proposition~\ref{inj}. Given a proper colouring $g:V(G)\to C_L$ and a colour $c\in C_L$, we say that $g$ \emph{agrees} with $h$ at $c$ if $g(h(c))=c$. 

Now, let $g:V(G)\to C_L$ be a proper colouring in which every vertex $v\in V(G)$ satisfies either \ref{ggood} or \ref{gcontained} and, subject to this, the number of colours $c\in C_L$ at which $g$ agrees with $h$ is maximized. We show that $g$ is surjective. Otherwise, let $c'\in C_L-g(V(G))$ be arbitrary and define a colouring $g':V(G)\to C_L$ as follows:
\[g'(v)=\left\{\begin{array}{ll}	c' 		& \text{if }v=h(c'),\\
																	g(v)	&	\text{otherwise}.\end{array}\right.\]
																	
Clearly $g'$ is proper since $g$ does not map any vertex to $c'$. Moreover, $g'$ agrees with $h$ at $c'$ and at every colour at which $g$ agrees with $h$. Let us show that every vertex $v$ of $G$ satisfies either \ref{ggood} or \ref{gcontained} for $g'$, which will contradict our choice of $g$ and complete the proof.

In the case that $v=h(c')$, then we have $g'(v)=c'\in L(v)$ and so \ref{ggood} is satisfied for $v$. Now, suppose that $v\neq h(c')$ and $g'(v)\notin L(v)$. Since every vertex $w\neq h(c')$ satisfies $g'(w)=g(w)\neq c'$, we see that 
\[g'^{-1}(g'(v))=g^{-1}(g(v))-h(c')\subseteq f^{-1}(f(v))\]
and so \ref{gcontained} is satisfied for $v$. The result follows. 
\end{proof}

\section{Basic Properties of a Minimal Counterexample}
\label{mce} 

Our next goal is to use the minimality assumption on $G$ to obtain some useful properties of $L$. The following proposition describes a general situation in which we are able to do so.

\begin{prop}
\label{minimal}
For $\ell\geq1$, suppose that there is a set $A\subseteq V(G)$ and a proper colouring $g:A\to C_L$ of $G[A]$ such that
\begin{enumerate}[leftmargin=*,labelindent=1em,label=\rm{(\alph*)}]
\item \label{2ell} $|V(G)-A|\leq 2(k-\ell)+1$,
\item \label{k-ell} $\chi(G-A)\leq k-\ell$,
\item \label{gacc} $g(v)\in L(v)$ for every $v\in A$, and
\item \label{ellL} $|g(A)\cap L(w)|\leq \ell$ for $w\in V(G)-A$.
\end{enumerate}
Then there is an acceptable colouring for $L$. 
\end{prop}

\begin{proof}
Let $G'$ be a graph on $2(k-\ell)+1$ vertices with chromatic number $k-\ell$ obtained by adding a (possibly empty) set of vertices and edges to $G-A$. Clearly, $G'$ satisfies the hypothesis of Ohba's Conjecture and so, by minimality of $G$, we have that $G'$ is $(k-\ell)$-choosable. We let $L'$ be any list assignment of $G'$ with the following properties:
\begin{enumerate}[leftmargin=*,labelindent=1em,label=$\bullet$]
\item for each $w\in V(G)-A$, we have $L'(w)=L(w)-g(A)$, and 
\item for each $w\in V(G')-V(G)$ we have $|L'(w)|\geq k-\ell$.
\end{enumerate}

Given $w\in V(G)-A$, we have
\[|L'(w)|=|L(w)|-|g(A)\cap L(w)|\geq k-\ell.\]
Therefore, since $G'$ is $(k-\ell)$-choosable, there exists an acceptable colouring $f'$ for $L'$. We obtain an acceptable colouring $f$ for $L$ by colouring each vertex $v\in A$ with $g(v)$ and each vertex $v\in V(G)-A$ with $f'(v)$. The result follows. 
\end{proof}

Next we prove a simple, yet important, corollary of Proposition~\ref{minimal}. One should compare the following result to Lemma~\ref{emptycap}.

\begin{cor}
\label{cap}
If $P$ is a part of $G$ such that $|P|\geq 2$, then $\cap_{v\in P}L(v)=\emptyset$.
\end{cor}

\begin{proof}
Otherwise, let $c\in \cap_{v\in P}L(v)$ be arbitrary and define $g(v):=c$ for all $v\in P$. Then clearly the set $A:=P$ and the function $g$ satisfy Proposition~\ref{minimal} for $\ell=1$, a contradiction.  
\end{proof}

Combining Observation~\ref{<n} and Corollary~\ref{cap}, we obtain the following.

\begin{cor}
\label{allC}
If $P=\{u,v\}$ is a part of $G$, then $C_L=L(u)\cup L(v)$ and $|C_L|=2k$.
\end{cor}

Using Corollary~\ref{allC}, we see that $|V(G)|$ is exactly $2k+1$. 

\begin{cor}
$|V(G)|=2k+1$. 
\end{cor}

\begin{proof}
By assumption, we have $|V(G)|\leq 2k+1$. If $G$ contains a part of size $2$, then by Corollary~\ref{allC} we have $|V(G)|>|C_L|\geq 2k$, and so the result holds in this case. 

On the other hand, suppose that $|V(G)|\leq 2k$ and that $G$ does not contain a part of size $2$. Then $G$ must contain a singleton, say $v$. Choose a colour $c\in L(v)$ arbitrarily and define $g(v):=c$. Then, since $|V(G)|\leq 2k$ and $v$ is a singleton, we have that $A:=\{v\}$ and the function $g$ satisfy Proposition~\ref{minimal} for $\ell=1$. This contradiction completes the proof.
\end{proof}

\section{Colouring Outside of the Lists}
\label{outside}

To present the proof of Lemma~\ref{na}, we require a generalization of the bipartite graph $B_L$ used in the proof of Theorem~\ref{222}. 

\begin{defn}
Given a proper colouring $f$ of $G$, let 
\begin{enumerate}[leftmargin=*,labelindent=1em,label=$\bullet$]
\item $V_f:=\left\{f^{-1}(c):c\in f(V(G))\right\}$,
\item $B_{f,L}$ be the bipartite graph with bipartition $\left(V_f,C_L\right)$ where each $f^{-1}(c)\in V_f$ is joined to the colours of $\cap_{v\in f^{-1}(c)}L(v)$. 
\end{enumerate}
\end{defn}

Notice that, if there is a matching $M$ in $B_{f,L}$ which saturates $V_f$, then we can obtain an acceptable colouring for $L$ by mapping every vertex of $f^{-1}(c)$ to the colour that $f^{-1}(c)$ is matched to under $M$. In proving Lemma~\ref{na}, our aim is to show that such a matching exists whenever $f$ is near-acceptable for $L$. 

\begin{proof}[Proof of Lemma~\ref{na}]
Let $f$ be a near-acceptable colouring for $L$. By the above discussion, we can assume that there is no matching in $B_{f,L}$ which saturates $V_f$. So, by Hall's Theorem, there exists a set $S\subseteq V_f$ such that $\left|N_{B_{f,L}}(S)\right|<|S|$. In particular, this implies that there is a colour $c^*\in C_L$ such that 
\begin{equation}\label{notinS}f^{-1}(c^*)\in S,\text{ and}\end{equation}
\begin{equation}\label{cN(S)}c^*\notin N_{B_{f,L}}(S).\end{equation}
Combining (\ref{notinS}) and (\ref{cN(S)}), we see that $c^*\notin \cap_{v\in f^{-1}(c^*)}L(v)$ and so there is a vertex $v^*$ such that $f(v^*)=c^*$ and $c^*\notin L(v^*)$. Since $f$ is near-acceptable for $L$, this implies that
\[f^{-1}(c^*)=\{v^*\},\text{ and}\]
\[c^*\text{ is frequent.}\]
From this point forward, we divide the proof into two cases. 

\begin{case}
$c^*$ is globally frequent. 
\end{case}

In this case, since $f^{-1}(c^*)=\{v^*\}\in S$, we have $L(v^*)\subseteq N_{B_{f,L}}(S)$ and so $\left|N_{B_{f,L}}(S)\right|\geq k$. Thus, by our choice of $S$ we have
\begin{equation}\label{atleastk}|S|\geq \left|N_{B_{f,L}}(S)\right| +1\geq k+1.\end{equation}

Now, recall that $c^*\notin N_{B_{f,L}}(S)$. It follows that every colour class $f^{-1}(c)\in S$ contains a vertex $w$ for which $c^*\notin L(w)$. Thus, the cardinality of $S$ is at most the number of vertices of $G$ for which $c^*$ is not available. However, since $c^*$ is globally frequent, we have that $c^*$ is available for at least $k+1$ vertices of $G$. Therefore, 
\[|S|\leq |V(G)|-(k+1)\leq (2k+1)-(k+1)=k\]
contradicting (\ref{atleastk}). This completes the proof in this case. 

\begin{case}
\label{freqamong}
$c^*$ is frequent among singletons. 
\end{case}

In this case, the proof is more complicated, but the underlying idea is straightforward. Our objective is to find a set $A\subseteq V(G)$ and function $g:A\to C_L$ which satisfy the conditions of Proposition~\ref{minimal} for some integer $\ell$. This will imply that there is an acceptable colouring for $L$, completing the proof.

First, we require some stronger assumptions on the colouring $f$ and the set $S$. By Proposition~\ref{surj}, we can assume that $f$ maps surjectively to $C_L$. Also, we let $S$ be chosen to maximize $|S|-\left|N_{B_{f,L}}(S)\right|$ over all subsets of $V_f$. By the choice of $S$, we obtain the following:

\begin{claim}
There is a matching $M'$ in $B_{f,L}-N_{B_{f,L}}(S)$ which saturates $V_f-S$. 
\end{claim}

\begin{proof}
Otherwise, by Hall's Theorem, there is a set $T\subseteq V_f-S$ such that 
\[|T|>\left|N_{B_{f,L}}(T)-N_{B_{f,L}}(S)\right|.\]
Now, if we define $S'=S\cup T$, we see that
\[|S'|-\left|N_{B_{f,L}}(S')\right| = |S|+|T|- \left|N_{B_{f,L}}(S)\right|- \left|N_{B_{f,L}}(T)-N_{B_{f,L}}(S)\right| \]
\[>|S|-\left|N_{B_{f,L}}(S)\right|\]
contradicting our choice of $S$. The result follows. 
\end{proof}

Next, we describe the choice of $\ell,A$ and $g$. Let $\ell$ denote the number of colour classes in $V_f$ containing more than one element. We define $A$ to be the union of colour classes of $f^{-1}(c)$ such that either 
\begin{enumerate}[leftmargin=*,labelindent=1em,label=$\bullet$]
\item  $f^{-1}(c)$ is contained in $V_f-S$, or 
\item $f^{-1}(c)$ contains more than one element.
\end{enumerate}
Given this, we define $g:A\to C_L$ so that, for each colour class $f^{-1}(c)\subseteq A$, every vertex of $f^{-1}(c)$ is mapped to the same colour, chosen in the following way:
\begin{enumerate}[leftmargin=*,labelindent=1em,label=$\bullet$]
\item if $f^{-1}(c)\in V_f-S$, then for each $v\in f^{-1}(c)$ we set $g(v)$ to be the colour which is matched to $f^{-1}(c)$ under $M'$, 
\item if $f^{-1}(c)\in S$, then for each $v\in f^{-1}(c)$ we set $g(v)=c$. Note that, by definition of $A$, every such colour class must satisfy $\left|f^{-1}(c)\right|\geq 2$. So, since $f$ is near-acceptable for $L$, we have $c\in L(v)$ for every $v\in f^{-1}(c)$.
\end{enumerate}
Note that any two distinct colour classes of $f$ contained in $A$ are mapped by $g$ to different colours. Therefore, $g$ is a proper colouring of $G[A]$ and it satisfies property \ref{gacc} of Proposition~\ref{minimal} by definition. To complete the proof, we show that $A$ and $g$ satisfy properties \ref{2ell}, \ref{k-ell} and \ref{ellL} of Proposition~\ref{minimal}.

First, observe that $|A|\geq 2\ell$ since $A$ contains every non-singleton colour class of $f$. Therefore, $A$ satisfies property \ref{2ell} of Proposition~\ref{minimal}.  Now, for any singleton $x$ of $G$ for which $c^*$ is available, we have that $\{x\}$ is a colour class of $f$ which is not contained in $S$ (since $f$ is proper and $c^*\notin N_{B_L}(S)$). Therefore, $A$ contains every singleton for which $c^*$ is available. Since $c^*$ is frequent among singletons, this implies that 
\[\chi(G-A)\leq k-\gamma.\]
Recall that $f$ is surjective which, by the pigeonhole principle, implies that $\ell\leq \gamma$. Therefore, we have that $k-\gamma\leq k-\ell$, and so $A$ satisfies property \ref{k-ell} of  Proposition~\ref{minimal}. 

So, all that remains is prove that $g$ satisfies property \ref{ellL} of Proposition~\ref{minimal}. First, notice that for every vertex $w\in V(G)-A$ we have that $\{w\}$ is a colour class of $f$ which is contained in $S$ by definition of $A$. This implies that $L(w)\subseteq N_{B_L}(S)$ for every such vertex $w$. Therefore, to show that $g$ satisfies property \ref{ellL} of Proposition~\ref{minimal}, it suffices to prove the following:
\begin{equation}\label{gl}\left|g(A)\cap N_{B_L}(S)\right|\leq \ell.\end{equation}

To see that (\ref{gl}) is true, recall that $g$ maps $f^{-1}(c)\subseteq A$ to a colour of $N_{B_L}(S)$ if and only if $f^{-1}(c)\in S$. Therefore, $\left|g(A)\cap N_{B_L}(S)\right|$ is bounded above by the number of colour classes $f^{-1}(c)\in S$ contained in $A$. However, by definition of $A$, every colour class $f^{-1}(c)\in S$ contained in $A$ must contain more than one element. The number of colour classes of $f$ containing more than one element is precisely $\ell$, and so (\ref{gl}) holds. This completes the proof of Lemma~\ref{na}. 
\end{proof}

\section{The Greedy Colouring Procedure}
\label{procSec}

To prove Lemma~\ref{na}, we showed that a near-acceptable colouring for $L$ can be modified to produce an acceptable colouring for $L$. Our next goal is to prove Lemma~\ref{colproc}, which says that there exists a near-acceptable colouring for $L$ provided that there are at least $k$ frequent colours. 

\begin{proof}[Proof of Lemma~\ref{colproc}]
Let $F$ be a set of exactly $k$ frequent colours. Our objective is to construct a near-acceptable colouring for $L$ via a three phase greedy procedure. For $i\in\{1,2,3\}$, the $i$th phase of the procedure involves colouring a set $V_i$ of vertices with a set $C_i$ of colours. The key feature of this procedure is that it always produces a proper colouring $f$ of $G[V_1\cup V_2\cup V_3]$ such that if $V_1\cup V_2\cup V_3=V(G)$, then $f$ is near-acceptable for $L$. We describe the procedure now.

\begin{phase}
\label{P1}
Define $C_1:=C_L-F$. We choose a subset $V_1$ of $V(G)$ and a mapping $g_1:V_1\to C_1$ to have the following properties:
\begin{enumerate}[leftmargin=*,labelindent=1em,label=(P1.\arabic*)]
\item\label{good1} $g_1$ is a proper colouring of $G[V_1]$ and for each $v\in V_1$ we have $g_1(v)\in L(v)$,
\item\label{max1} subject to \ref{good1}, $|V_1|$ is maximized,
\item\label{maxpart1} subject to \ref{good1} and \ref{max1}, the number of parts $P$ of $G$ such that $P\cap V_1\neq \emptyset$ is maximized. 
\end{enumerate}
\end{phase}

We observe the following:

\begin{obs}
\label{stillk+1}
We can assume that $\left|V_1\right|\leq k$. Otherwise, we would have $|V(G)-V_1|\leq (2k+1) - (k+1)= k$. Since $|F|=k$ and the colours of $F$ are frequent, we could obtain a near-acceptable colouring for $L$ by simply mapping the vertices of $V(G)-V_1$ to distinct colours of $F$ in an arbitrary way. Thus, the proof is complete unless $\left|V_1\right|\leq k$.
\end{obs}

After completing Phase~\ref{P1}, for each part $P$ of $G$ we let $P':=P-V_1$. We label the parts of $G$ by $P_1,\dots,P_k$ so that
\[\left|P_1'\right|\geq \dots \geq\left|P_k'\right|.\]
Given this ordering, we colour the vertices of $V(G)-V_1$ in the following way.

\begin{phase}
\label{P2}
For each $i=1,\dots,k$, in turn, if there is a colour $c_i\in F-\{c_{i'}:i'<i\}$ which is available for every vertex of $P_i'$, then we set $g_2(v):=c_i$ for every vertex $v\in P_i'$. We terminate Phase~\ref{P2} when we reach an index $i$ for which no such $c_i$ exists. Let $V_2$ be the set of vertices which are coloured during Phase~\ref{P2} and let $C_2:=g_2(V_2)$.
\end{phase}

We remark that every vertex of $V_1\cup V_2$ is mapped by either $g_1$ or $g_2$ to a colour which is contained in its list. That is, until this point of the procedure, the colouring which we have constructed is acceptable for $L$. In the final phase, we allow vertices to be mapped to frequent colours that are outside of their lists.

\begin{phase}
\label{P3}
Define $C_3 := F-C_2$. For each colour $c\in C_3$, we colour at most one vertex of  $V(G)-V_1-V_2$ with $c$, where this vertex is chosen arbitrarily. Let $V_3$ be the set of vertices which are coloured in Phase~\ref{P3}. 
\end{phase}

Clearly, if $V_1\cup V_2\cup V_3=V(G)$, then the resulting colouring is near-acceptable for $L$. Since, in Phase~\ref{P3}, we are free to map each vertex of $V(G)-V_1-V_2$ to an arbitrary colour of $C_3$, we are done if the following inequality holds:
\begin{equation}\label{colourenuf}|V(G)-V_1-V_2|\leq |C_3|.\end{equation}
In order to show that (\ref{colourenuf}) holds, we require the following two claims; we state them now and defer their proofs to the end.

\begin{claim}
\label{j}
For any part $P$ of $G$, we have
\[\sum_{v\in P'}|L(v)\cap C_1|\leq |V_1-P|.\]
\end{claim}

\begin{claim}
\label{3hit}
If $P$ is a non-singleton part of $G$, then either $P\cap V_1\neq\emptyset$ or there is a near-acceptable colouring for $L$. 
\end{claim}

For now, we assume that Claims~\ref{j} and~\ref{3hit} are true. By Claim~\ref{3hit}, we are done unless
\begin{equation}\label{hittt} P\cap V_1\neq\emptyset\text{ for every non-singleton part } P \text{ of } G.\end{equation}
Given (\ref{hittt}), our goal is to prove that (\ref{colourenuf}) holds. 

By definition of Phase~\ref{P2}, we can let $i\geq0$ be the index such that $V_2=\cup_{j=1}^{i}P_j'$. That is, in Phase~\ref{P2}, we have coloured every vertex of $\cup_{j=1}^iP_j'$, but no vertex of $\cup_{j=i+1}^kP_j'$. Note that if $i=k$, then by combining $g_1$ and $g_2$ we obtain an acceptable colouring for $L$, and so we can assume $i<k$. By definition of Phase~\ref{P2}, we have $|C_2|=i$ and so, since $|F|=k$, we must have $|C_3|=k-i$. Therefore, (\ref{colourenuf}) holds unless
\begin{equation}\label{k-i+1}\left|V(G)-V_1-V_2\right|=\sum_{j=i+1}^k\left|P_j'\right|\geq |C_3|+1 = k-i+1.\end{equation}
In particular, since $|V(G)|=2k+1$, we see that 
\begin{equation}\label{P12}\left|V_1\cup V_2\right|\leq k+i.\end{equation}
Also, by (\ref{k-i+1}) and our choice of ordering of the parts, we have
\[(k-i)\left|P_{i+1}'\right|\geq \sum_{j=i+1}^k\left|P_j'\right| \geq k-i+1\]
and so, 
\begin{equation}\label{i+12}\left|P_{i+1}'\right|\geq \left\lceil\frac{k-i+1}{k-i}\right\rceil =2.\end{equation} 

Now, again by our choice of ordering, we see that (\ref{i+12}) implies that $\left|P_j'\right|\geq 2$ for all $j\leq i$. Therefore, we have that $|V_2|= \sum_{j=1}^i\left|P_j'\right| \geq 2i$. Combining this with (\ref{P12}), we obtain
\begin{equation}\label{P1col}|V_1|\leq k-i.\end{equation}

By (\ref{hittt}) and (\ref{i+12}), we can assume that $V_1\cap P_{i+1}\neq\emptyset$ and so, by (\ref{P1col}),
\[\left|V_1-P_{i+1}\right| = \left|V_1\right|-\left|V_1\cap P_{i+1}\right|\leq \left|V_1\right|-1\leq k-i-1.\] Therefore, by Claim~\ref{j}, 
\begin{equation}\label{nonfreq}\sum_{v\in P_{i+1}'}\left|L(v)\cap C_1\right|\leq k-i-1.\end{equation}

Next, notice that each colour $c\in C_3$ must be absent from the list of at least one vertex in $P_{i+1}'$ for, if not, then we would have coloured the vertices of $P_{i+1}'$ with $c$ during Phase~\ref{P2}. Combining this fact with (\ref{nonfreq}), we see that
\[\sum_{v\in P_{i+1}'}|L(v)| = \sum_{j=1}^3\sum_{v\in P_{i+1}'}|L(v)\cap C_j|\leq (k-i-1) + \left|P_{i+1}'\right|\left|C_2\right| + \left(\left|P_{i+1}'\right|-1\right)\left|C_3\right|.\]
Now, substituting $|C_2|=i$ and $|C_3|=k-i$ into the inequality above gives us
\[\sum_{v\in P_{i+1}'}|L(v)| \leq (k-i-1) +\left|P_{i+1}'\right|i +  \left(\left|P_{i+1}'\right|-1\right)(k-i) = k\left|P_{i+1}'\right|-1.\]
However, this contradicts the fact that every list has size at least $k$. Thus, we have that (\ref{colourenuf}) holds. 

So, to complete the proof of Lemma~\ref{colproc}, all that remains is to prove Claims~\ref{j} and~\ref{3hit}. To do so, we apply properties \ref{max1} and \ref{maxpart1} of $V_1$ and $g_1$. We begin by proving Claim~\ref{j}.

\begin{proof}[Proof of Claim~\ref{j}]
To prove the claim, we will actually show the following: if a colour $c\in C_1$ is available for a set $T$ of $j>0$ vertices in $P'$, then 
\begin{enumerate}[leftmargin=*,labelindent=1em,label=$\bullet$]
\item $g_1^{-1}(c)\cap P=\emptyset$, and 
\item $\left|g_1^{-1}(c)\right|\geq j$.
\end{enumerate}
Clearly, this will imply the claim.

First, since $G$ is a complete multipartite graph, if $g_1^{-1}(c)\cap P\neq\emptyset$, then it must be the case that $g_1^{-1}(c)\subseteq P$. In this case, we can extend our colouring $g_1$ by mapping the vertices of $T$ to $c$, contradicting property \ref{max1}. Similarly, if $g_1^{-1}(c)$ contains fewer than $j$ vertices of $V(G)-P$, then we can uncolour the vertices of $g_1^{-1}(c)$ and, instead, map the vertices of $T$ to $c$, again contradicting property \ref{max1}. The result follows. 
\end{proof}

We prove a special case of Claim~\ref{3hit} separately before establishing the full result.

\begin{claim}
\label{2hit}
If $P=\{u,v\}$ is a part of $G$, then $P\cap V_1\neq\emptyset$. 
\end{claim}

\begin{proof}
Suppose to the contrary that $P\cap V_1=\emptyset$. By Corollary~\ref{allC}, we see that 
\begin{equation}\label{allll}L(u)\cup L(v)=C_L\text{ and}\end{equation} \begin{equation}\label{2kkk}\left|C_L\right|=2k. \end{equation}

Before moving on, let us prove the following:
\begin{equation} \label{1} \left|g_1^{-1}(c)\right|=1 \text{ for every }c\in C_1.\end{equation}

First, we observe that $g_1^{-1}(c)$ cannot be empty for any colour $c\in C_1$. Otherwise, since $P\cap V_1=\emptyset$ and $L(u)\cup L(v)=C_L$, we could colour one of $u$ or $v$ with $c$, increasing the number of coloured vertices and contradicting \ref{max1}. Therefore, $\left|g_1^{-1}(c)\right|\geq 1$ for all $c\in C_1$. Now, suppose that $\left|g^{-1}(c)\right|\geq 2$ for some $c\in C_1$. By (\ref{2kkk}), we have that $|C_1| =|C_L|-|F|=k$. Thus, we would have that
\[|V_1|=\sum_{c'\in C_1} \left|g^{-1}(c')\right| \geq \left|g^{-1}(c)\right| + |C_1-c| \geq k+1\]
contradicting Observation~\ref{stillk+1}. Therefore, (\ref{1}) holds.

So, by (\ref{1}) and the fact that $|C_1|=k$, we see that $|V_1|=k$. Combining this with the assumption that $P\cap V_1=\emptyset$ and the fact that $G$ consists of exactly $k$ parts, we see that there exists a part $Q$ of $G$ such that $\left|Q\cap V_1\right|\geq 2$. Let $x$ be any vertex of $Q\cap V_1$, and define $c:=g_1(x)$. 

By (\ref{allll}) and without loss of generality, we have $c\in L(u)$. Let $g_1'$ be a partial colouring of $G$ obtained from $g_1$ by uncolouring $x$ and colouring $u$ with $c$. Since (\ref{1}) implies that $g_1^{-1}(c)=\{x\}$, we see that $g_1'$ is proper. However, $g_1$ and $g_1'$ colour the same number of vertices, namely $k$, but there are more parts of $G$ which contain a vertex that is coloured under $g_1'$ than under $g_1$. This contradicts \ref{maxpart1} and completes the proof. 
\end{proof}

Finally, we prove Claim~\ref{3hit}, and thus complete the proof of Lemma~\ref{colproc}.

\begin{proof}[Proof of Claim~\ref{3hit}]
Suppose that there is a non-singleton part $P$ of $G$ such that $P\cap V_1=\emptyset$. Then we must have $|P|\geq 3$ by Claim~\ref{2hit}. We prove that there exists a near-acceptable colouring for $L$.

By Corollary~\ref{cap}, we have that every colour $c\in F$ is be available for at most $|P|-1$ vertices of $P$. So, in total, 
\begin{equation}\label{freqs}\sum_{v\in P}|L(v)\cap F|\leq |F|(|P|-1)=k(|P|-1).\end{equation}
By Observation~\ref{stillk+1} we have $|V_1|\leq k$ and so since $P\cap V_1=\emptyset$, we obtain the following by Claim~\ref{j}:
\begin{equation}\label{infreqs}\sum_{v\in P}|L(v)-F| = \sum_{v\in P}|L(v)\cap C_1| \leq |V_1|\leq k.\end{equation}
Combining (\ref{freqs}) and (\ref{infreqs}), we have
\[\sum_{v\in P}|L(v)| = \sum_{v\in P}|L(v)\cap F|+\sum_{v\in P}|L(v)-F|\leq k(|P|-1) + k = k|P|.\]
However, since every list has size at least $k$, we see that both (\ref{freqs}) and (\ref{infreqs}) must be tight. Therefore,
\begin{equation}\label{2F}\text{each colour }c\in F\text{ is available for exactly }|P|-1\geq 2\text{ vertices of }P,\text{ and}\end{equation}
\begin{equation}\label{infreqapp}\text{the colours of }C_1\text{ appear exactly }k\text{ times in the lists of vertices of }P.\end{equation}

However, by applying Claim~\ref{j}, we see that (\ref{infreqapp}) implies that $|V_1|=k$. In particular, we have $|V(G)-V_1|=k+1$. Now, we extend $g_1$ to a near-acceptable colouring for $L$ as follows:
\begin{enumerate}[leftmargin=*,labelindent=1em,label=$\bullet$]
\item let $c\in F$ be arbitrary and use $c$ to colour a set $T$ of $|P|-1\geq 2$ vertices of $P$ for which $c$ is available. This is possible by (\ref{2F}). 
\item Since $|T|\geq2$ and $P\cap V_1=\emptyset$, we see that $|V(G)-V_1-T|\leq k-1$. Thus, since $|F-c|=k-1$, we can colour the vertices of $V(G)-V_1-T$ with distinct colours of $F-c$ arbitrarily. 
\end{enumerate}
This completes the proof of Claim~\ref{3hit} and of Lemma~\ref{colproc}.
\end{proof}
\end{proof}

\section{Colours Common to the Singletons}
\label{commontosing}

Our final goal is to prove Lemma~\ref{enough}, from which the main result will follow. From this point forward, we assume to that $L$ is a list assignment as in the Colour Matching Lemma and that
\begin{enumerate}[leftmargin=*,labelindent=1em,label=$\bullet$]
\item $C_L$ contains fewer than $k$ frequent colours.
\end{enumerate}
Additionally, we can assume that $L$ is \emph{maximal} in the following sense:
\begin{enumerate}[leftmargin=*,labelindent=1em,label=$\bullet$]
\item there does not exist an acceptable colouring for $L$, but if $v\in V(G)$ and $c\notin L(v)$, then there is an acceptable colouring for the list assignment $L'$ where $L'(v):=L(v)\cup\{c\}$ and $L'(w):=L(w)$ for every vertex $w\neq v$. 
\end{enumerate}
Given that $L$ is maximal, it follows easily that every frequent colour is available for every singleton.

\begin{lem}
\label{freq=>sing}
If $c\in C_L$ is frequent, then $c\in L(v)$ for every singleton $v$.
\end{lem}

\begin{proof}
Otherwise, add $c$ to the list of $v$. Since $L$ is maximal, we obtain an acceptable colouring $f$ for this modified list assignment. However, since $f$ is proper and $v$ is a singleton, we have $f^{-1}(f(v))=\{v\}$. Since $c$ is frequent, it follows that $f$ is a near-acceptable colouring for $L$. By Lemma~\ref{na}, this implies that there is an acceptable colouring for $L$, which is a contradiction.
\end{proof}

Our next goal is to prove the converse of Lemma~\ref{freq=>sing}. To do so, it will be useful to keep track of the number of singleton and non-singleton parts of $G$. 

\begin{defn}
Let $p$ denote the number of non-singleton parts in $G$.
\end{defn}

\begin{obs}
$G$ contains precisely $k-p$ singletons. 
\end{obs}

To prove the converse of Lemma~\ref{freq=>sing}, we require a preliminary bound on the number of globally frequent colours. This bound will be used again in the next section.

\begin{lem}
\label{global}
If $F'$ is the set of all globally frequent colours in $C_L$, then
\[|F'|\geq\frac{k\gamma}{k+1-p}.\]
\end{lem}

\begin{proof}
By Lemma~\ref{cap}, each colour of $F'$ can be available for at most $|V(G)|-p = 2k+1-p$ vertices of $G$. Also, since colours of $C_L-F'$ are not globally frequent, they can be available for at most $k$ vertices of $G$. Therefore,
\[k|V(G)|\leq \sum_{v\in V(G)}|L(v)| = \sum_{c\in C_L}\left|N_{B_L}(c)\right|\leq (2k+1-p)\left|F'\right|+ k\left|C_L-F'\right|\]
\[ = (k+1-p)\left|F'\right|+k\left|C_L\right|.\]
Rearranging, we obtain
\[\left|F'\right|\geq\frac{k\left(|V(G)|-\left|C_L\right|\right)}{k+1-p} = \frac{k\gamma}{k+1-p}\]
as desired. 
\end{proof}

\begin{cor}
\label{gammasing}
$G$ contains at least $\gamma$ singletons. 
\end{cor}

\begin{proof}
Otherwise, we would have $\gamma\geq k+1-p$. However, by Lemma~\ref{global}, this would imply that there are at least $k$ frequent colours, contradicting our choice of $L$. 
\end{proof}

\begin{cor}
\label{freqsing}
A colour $c\in C_L$ is frequent if and only if it is available for every singleton.
\end{cor}

\begin{proof}
By Corollary~\ref{gammasing}, there are at least $\gamma$ singletons and so any colour that is available for every singleton is frequent among singletons. For the converse, we apply Lemma~\ref{freq=>sing}. The result follows. 
\end{proof}

We obtain another consequence of Corollary~\ref{gammasing}. 

\begin{cor}
\label{no2}
$G$ does not contain a part of size $2$.
\end{cor}

\begin{proof}
If $G$ contains a part of size $2$, then by Corollary~\ref{allC} we see that $|C_L|=2k$ and so $\gamma=1$. By Corollary~\ref{gammasing}, we have that $G$ contains a singleton, say $v$. Since $\gamma=1$, every colour of $L(v)$ is frequent among singletons, and so there are at least $k$ frequent colours. This contradicts our choice of $L$ and completes the proof.
\end{proof}

\begin{cor}
\label{k+1/2}
$p\leq\frac{k+1}{2}$. 
\end{cor}

\begin{proof}
By Corollary~\ref{no2}, we see that every part of $G$ is either a singleton, or has size at least $3$. Therefore,
\[(k-p) + 3p\leq |V(G)|= 2k+1\]
which implies $p\leq\frac{k+1}{2}$. 
\end{proof}

By applying Corollary~\ref{freqsing}, we obtain a strong restriction on the number of frequent colours. 

\begin{lem}
\label{<p}
$C_L$ contains fewer than $p$ frequent colours. 
\end{lem}

\begin{proof}
Otherwise, let $A_p=\left\{c_1,\dots,c_p\right\}$ be a set of $p$ frequent colours. By Corollary~\ref{freqsing}, 
\begin{equation}\label{Ap}\text{every colour of }A_p\text{ is available for every singleton.}\end{equation}

Now, label the singletons of $G$ by $v_{p+1},\dots,v_k$. For each $i=p+1,\dots, k$, in turn, choose a colour $c_i\in L(v_i)-A_{i-1}$ greedily and define $A_i:=A_{i-1}\cup\left\{c_i\right\}$. Let $L'$ be the list assignment of $G$ defined by
\[L'(v):=\left\{\begin{array}{ll}A_k & \text{if }v\text{ is a singleton},\\
L(v) & \text{otherwise.}\end{array}\right.\]
Clearly $|L'(v)|\geq k$ for all $v\in V(G)$. 

We argue that there exists an acceptable colouring for $L'$. If $C_{L'}\subsetneq C_L$, then this follows from the assumption that $L$ satisfies the hypotheses of the Colour Matching Lemma. So, we assume that $C_{L'}=C_L$, which implies that $|V(G)|-|C_{L'}|=\gamma$. However, under $L'$, we have that $A_k$ is a set of $k$ colours each of which is available for every singleton. So, by Corollary~\ref{gammasing}, every colour of $A_k$ is frequent among singletons with respect to $L'$. However, this implies that there is an acceptable colouring for $L'$ by Lemmas~\ref{na} and~\ref{colproc}.

Now, let $f'$ be an acceptable colouring for $L'$. We modify $f'$ to obtain an acceptable colouring $f$ for $L$ in the following way. For each vertex $v$ which is not a singleton, we simply set $f(v):=f'(v)$. Thus, all that remains is to colour the singletons. Note that $f'$ maps the singletons to a set $A'$ of exactly $k-p$ colours of $A_k$. We colour each singleton $v$ of $G$ with a colour of $L(v)\cap A'$ as follows:
\begin{enumerate}[leftmargin=*,labelindent=1em,label=$\bullet$]
\item For $i\geq p+1$, if $c_i\in A'$, then colour $v_i$ with $c_i$.
\item Otherwise, colour $v_i$ with a colour of $A'\cap A_p$ arbitrarily. 
\end{enumerate}
Since $|A'|=k-p$, which is precisely the number of singletons, we see that this is always possible. That is, $\left|A'\cap A_p\right|$ is precisely the number of singletons $v_i$ such that $c_i\notin A'$. We have that $f$ is an acceptable colouring for $L$, a contradiction. The result follows. 
\end{proof}

\section{Counting the Frequent Colours}
\label{count}

In the previous section, we obtained many useful properties of the frequent colours. All that is left now is to roll up our sleeves and \emph{count} them! The next proposition is key to our argument. While the statement is somewhat technical, we remark that it can be viewed as a strengthening of Lemma~\ref{na}. 

\begin{prop}
\label{technical}
Let $c^*$ be a colour which is not available for every singleton. Then there is a set $X(c^*)$ of singletons (depending on $c^*$) such that
\begin{enumerate}[leftmargin=*,labelindent=1em,label=\rm{(\alph*)}]
\item \label{bigset} $|X(c^*)|\geq k-p-\gamma+1$, and
\item \label{smallneigh}$\left|\cup_{v\in X(c^*)}L(v)\right|\leq 2k-\left|N_{B_L}(c^*)\right|$.
\end{enumerate}
\end{prop}

\begin{proof}
Let $v^*$ be a singleton for which $c^*$ is not available and for $v\in V(G)$ define
\[L^*(v):=\left\{\begin{array}{ll}L(v^*)\cup\{c^*\} & \text{if }v=v^*,\\
															L(v)							& \text{otherwise.}\end{array}\right.\]
By maximality of $L$, there exists an acceptable colouring $f^*$ for $L^*$. It must be the case that $f^*(v^*)=c^*$ as, otherwise, $f^*$ would be an acceptable colouring for $L$. Moreover, since $v^*$ is a singleton and $f^*$ is proper we have ${f^*}^{-1}(c^*)=\{v^*\}$.

Now, if there is a matching in $B_{f^*,L}$ which saturates $V_{f^*}$, then we can obtain an acceptable colouring for $L$. Thus, we assume that there is a set $S\subseteq V_{f^*}$ such that $\left|N_{B_{f^*,L}}(S)\right|<|S|$. Since every vertex $v\neq v^*$ has $f^*(v)\in L(v)$, we have that
\begin{equation}{f^*}^{-1}(c^*)\in S\text{, and}\end{equation}
\begin{equation}\label{starnotin}c^*\notin N_{B_{f^*,L}}(S).\end{equation}

Let $X(c^*)$ denote the set of all singletons $v$ such that $\{v\}\in S$. By (\ref{starnotin}), we see that every colour class in $S$ must contain a vertex whose list does not contain $c^*$. It follows that $|S|\leq 2k+1 - \left|N_{B_L}(c^*)\right|$. Therefore, we have
\[\left|\cup_{v\in X(c^*)}L(v)\right| \leq \left|N_{B_{f^*,L}}(S)\right|<|S|\leq 2k+1 - \left|N_{B_L}(c^*)\right|\]
which verifies \ref{smallneigh}.

Now, for every singleton $w\notin X(c^*)$, we have $\{w\}\in V_{f^*}-S$. If $|X(c^*)|\leq k-p-\gamma$, then there there are at least $\gamma$ such singletons and so we can proceed as in Case~\ref{freqamong} of the proof of Lemma~\ref{na} to obtain an acceptable colouring for $L$. Therefore, \ref{bigset} must hold. This completes the proof. 
\end{proof}

The advantage of Proposition~\ref{technical} is that it gives us a relatively \emph{large} set $X(c^*)$ of singletons whose lists are contained in a \emph{small} set of colours. Thus, an average colour of $\cup_{v\in X(c^*)}L(v)$ is available for many vertices of $X(c^*)$. Using this, we will show that there are at least $p$ colours of $\cup_{v\in X(c^*)}L(v)$ which are available for at least $\gamma$ singletons in $X(c^*)$, contradicting Lemma~\ref{<p} and completing the proof of Ohba's Conjecture. To make this more precise, we require several definitions. 

\begin{defn}
\label{star}
Let $c^*$ be a colour which is not available for every singleton and, subject to this, let $\left|N_{B_L}(c^*)\right|$ be maximum. 
\end{defn}

\begin{defn}
Let $Z$ be the set of $p-1$ colours which appear most frequently in the lists of vertices of $X(c^*)$.
\end{defn}

\begin{obs}
\label{theobs}
We can assume that every frequent colour is contained in $Z$ since (1) $X(c^*)$ consists of singletons, (2) by Corollary~\ref{freqsing}, every frequent colour is available for every singleton, and (3) there are fewer than $p$ frequent colours by Lemma~\ref{<p}. 
\end{obs}

\begin{defn}
Define $Y:=\cup_{v\in X(c^*)}L(v)-Z$ and let $c'$ be a colour in $Y$ such that $\left|N_{B_L}(c')\cap X(c^*)\right|$ is maximized. 
\end{defn}

To complete the proof of Lemma~\ref{enough} and of Ohba's Conjecture, we will show that that $\left|N_{B_L}(c')\cap X(c^*)\right|\geq \gamma$ which implies that $c'$ is frequent among singletons and contradicts Observation~\ref{theobs}. As we show next, it is enough to prove that $\left|N_{B_L}(c^*)\right|$ is relatively large.

\begin{prop}
\label{thismuch}
If $\left|N_{B_L}(c^*)\right|\geq k -\frac{(k-p-2\gamma+1)(k-p+1)}{\gamma}$, then $\left|N_{B_L}(c')\cap X\right|\geq\gamma$.
\end{prop}

\begin{proof}
For each vertex $v\in X(c^*)$ we have $L(v)\subseteq X\cup Y$, and so $|L(v)\cap Y|\geq k-|Z|=k-p+1$. Therefore, by our choice of $c'$,
\[|Y|\left|N_{B_L}(c')\cap X(c^*)\right|\geq \sum_{c\in Y}\left|N_{B_L}(c)\cap X(c^*)\right|=\sum_{v\in X(c^*)}|L(v)\cap Y|\geq |X(c^*)|(k-p+1).\]
Recall that $X(c^*)$ satisfies both bounds of Proposition~\ref{technical}. Therefore,
\[\left|N_{B_L}(c')\cap X(c^*)\right|\geq\frac{|X(c^*)|(k-p+1)}{|Y|}=\frac{|X(c^*)|(k-p+1)}{\left|\cup_{v\in X(c^*)}L(v)\right|-p+1}\] 
\[\geq \frac{(k-p-\gamma+1)(k-p+1)}{2k-\left|N_{B_L}(c^*)\right|-p+1}.\]
If $\left|N_{B_L}(c^*)\right|\geq k -\frac{(k-p-2\gamma+1)(k-p+1)}{\gamma}$, then we obtain
\[\left|N_{B_L}(c')\cap X(c^*)\right|\geq \frac{\gamma(k-p-\gamma+1)(k-p+1)}{\gamma(k-p+1)+(k-p-2\gamma+1)(k-p+1)}=\gamma,\]
as desired.
\end{proof}

Therefore, by Proposition~\ref{thismuch}, we need only show that $\left|N_{B_L}(c^*)\right|$ is at least $k -\frac{(k-p-2\gamma+1)(k-p+1)}{\gamma}$. To do so, we will first establish a preliminary bound on $\left|N_{B_L}(c^*)\right|$ and then verify that it is indeed as large as we require. 

\begin{prop}
\label{prelimbound}
$\left|N_{B_L}(c^*)\right|\geq k -\frac{(p-1)(k-p+1)-k\gamma}{2k-\gamma-p+2}$.
\end{prop}

\begin{proof}
Let $F$ denote the set of  frequent colours in $C_L$. Recall, that a colour is frequent if and only if it is available for every singleton. Therefore, it is equivalent to say that $c^*$ was chosen to maximize $\left|N_{B_L}(c^*)\right|$ over all colours which are not frequent. Thus, for every $c\notin F$ we must have $\left|N_{B_L}(c)\right|\leq\left|N_{B_L}(c^*)\right|$. Moreover, by Lemma~\ref{<p} there are at most $p-1$ frequent colours. If $c$ is frequent, then $\left|N_{B_L}(c)\right|\leq 2k+1-p$ by Lemma~\ref{cap}. Putting this together, we have
\[(2k+1)k\leq \sum_{v\in V(G)}|L(v)|=\sum_{c\in C_L}\left|N_{B_L}(c)\right|\leq \left|C_L-F\right|\left|N_{B_L}(c^*)\right|+|F|(2k+1-p)\]
\[\leq (2k+1-\gamma-p+1)\left|N_{B_L}(c^*)\right| + (p-1)(2k+1-p).\]
Solving for $\left|N_{B_L}(c^*)\right|$, we obtain
\[\left|N_{B_L}(c^*)\right|\geq\frac{(2k+1)k-(p-1)(2k+1-p)}{2k-\gamma-p+2}\]
\[=k+\frac{(\gamma+p-1)k-(p-1)(2k+1-p)}{2k-\gamma-p+2} = k+\frac{k\gamma - (p-1)(k+1-p)}{2k-\gamma-p+2}.\]
The result follows. 
\end{proof}

\begin{prop}
$\left|N_{B_L}(c^*)\right|\geq k -\frac{(k-p-2\gamma+1)(k-p+1)}{\gamma}$.
\end{prop}

\begin{proof}
By Proposition~\ref{prelimbound} it is enough to prove
\[k-\frac{(p-1)(k+1-p)-k\gamma}{2k-\gamma-p+2}\geq k-\frac{(k-p-2\gamma+1)(k+1-p)}{\gamma}\]
which is equivalent to
\begin{equation}\label{objective}\frac{(k-p-2\gamma+1)(k+1-p)}{\gamma}\geq \frac{(p-1)(k+1-p)-k\gamma}{2k-\gamma-p+2}.\end{equation}

By setting $\eta:=\frac{\gamma}{k+1-p}>0$, we can rewrite $\gamma$ as $\eta(k+1-p)$. Now, by dividing both sides of (\ref{objective}) by $k+1-p$ (which is positive) we see that (\ref{objective}) is equivalent to the following:
\begin{equation}\label{obj}\frac{1-2\eta}{\eta}\geq \frac{p-1 -\eta k}{(k+1)+(1-\eta)(k+1-p)}.\end{equation}

Corollary~\ref{k+1/2}, implies that $p-1<\frac{k}{2}$ and $k+1-p\geq\frac{k+1}{2}$. Thus, we obtain a bound on the right side of (\ref{obj}).
\begin{equation}\label{bingo}\frac{p-1 -\eta k}{(k+1)+(1-\eta)(k+1-p)}<\frac{\left(\frac{1}{2}-\eta\right)k}{\left(\frac{3}{2}-\frac{\eta}{2}\right)(k+1)}<\frac{1-2\eta}{3-\eta}.\end{equation}

Recall that Lemma~\ref{global} implies that there are at least $k\eta$ globally frequent colours. However, since there are at most $p-1$ frequent colours in total (Lemma~\ref{<p}), we see that $\eta k\leq p-1$. So, we have $\eta\leq\frac{p-1}{k}$ which by Corollary~\ref{k+1/2} is less than $\frac{1}{2}$. Therefore, $\frac{1-2\eta}{3-\eta}<\frac{1-2\eta}{\eta}$. Combining this with (\ref{bingo}), we obtain (\ref{obj}), which completes the proof of the proposition and of the main result. 
\end{proof}

Thus, Ohba's Conjecture is proved.

\begin{thm}[Noel et al.~\cite{NRW}]
\label{OCT}
If $|V(G)|\leq 2\chi(G)+1$, then $\ch(G)=\chi(G)$. 
\end{thm}

\let\newpage\relax
\let\clearpage\relax
}


{

{
\if@openright
    \cleardoublepage
  \else
    \clearpage
  \fi
\let\newpage\relax
\part{Beyond Ohba's Conjecture}
\label{beyond}
}

\let\newpage\relax
\let\clearpage\relax
{
\newpage
\clearpage
\addcontentsline{toc}{part}{Part~\ref{beyond}: Beyond Ohba's Conjecture}
\chapter{The Bigger Picture}
\label{glimpse}

\begin{chapquote}{Carl Gustav Jacob Jacobi}
One should always generalize.\end{chapquote}

\doublespacing 

In the rest of the thesis, we discuss several theorems and conjectures which are related, in one way or another, to Ohba's Conjecture. We first consider a direct strengthening of Ohba's Conjecture which provides a tight upper bound on the choice number of graphs on at most $3\chi$ vertices and improves a known result on the choice number of complete multipartite graphs with parts of size $4$. In the next chapter, we will provide a proof from~\cite{NWWZ} of the aforementioned theorem. We also propose two conjectures which generalize this theorem and relate it to a problem of Erd\H{o}s et al.~\cite{ERT} on the choice number of complete multipartite graphs. Finally, in the next section, we discuss an analog of Ohba's Conjecture for on-line choosability which was proposed by Huang et al.~\cite{online}.

As we have mentioned, Ohba's Conjecture is best possible in the sense that there are graphs on $2\chi+2$ vertices which satisfy $\ch>\chi$. However, many natural questions still remain. For example, we consider the following: \emph{what is the relationship between the choice number and the chromatic number for graphs on at most $3\chi$ vertices?} A natural starting point for this problem is a result of Kierstead~\cite{Kierstead} on the choice number of complete multipartite graphs in which every part has size $3$.

\begin{thm}[Kierstead~\cite{Kierstead}]
\label{3*k}
If $G$ is a complete multipartite graph in which every part has size $3$, then $\ch(G)=\left\lceil\frac{4\chi(G)-1}{3}\right\rceil$. 
\end{thm}

In the next chapter, we will present a proof of the following result of Noel, West, Wu and Zhu~\cite{NWWZ}, which implies that every graph $G$ on at most $3\chi(G)$ vertices satisfies $\ch(G)\leq\left\lceil\frac{4\chi(G)-1}{3}\right\rceil$; ie. the upper bound of Theorem~\ref{3*k} holds for all such graphs.

\begin{thm}[Noel et al.~\cite{NWWZ}]
\label{n/3}
For every graph $G$,
\[\ch(G)\leq\max\left\{\chi(G),\left\lceil\frac{|V(G)|+\chi(G)-1}{3}\right\rceil\right\}.\]
\end{thm}

Clearly, by Theorem~\ref{3*k}, we have that Theorem~\ref{n/3} is tight for complete multipartite graphs in which every part has size $3$. The following theorem of Ohba~\cite{3Ohba} provides us with a larger family of tight examples; we provide a proof of Ohba's result in Section~\ref{3and1}.

\begin{thm}[Ohba~\cite{3Ohba}]
\label{1and3}
Let $k_1$ and $k_3$ be integers and define $k:=k_1+k_3$ and $n:=k_1+3k_3$. If $G$ is the complete $k$-partite graph with $k_1$ parts of size $1$ and $k_3$ parts of size $3$, then
\[\ch(G)=\max\left\{k,\left\lceil\frac{n+k-1}{3}\right\rceil\right\}.\]
\end{thm}

Theorem~\ref{n/3} also implies that the difference between the choice number and the chromatic number of a graph $G$ is bounded in terms of the difference between $|V(G)|$ and $2\chi(G)+1$, as the following corollary demonstrates. From this, it is easily observed that Theorem~\ref{n/3} is a strengthening of Ohba's Conjecture.

\begin{cor}[Noel et al.~\cite{NWWZ}]
For every graph $G$,
\[\ch(G)-\chi(G)\leq\max\left\{0,\left\lceil\frac{|V(G)|-(2\chi(G)+1)}{3}\right\rceil\right\}.\]
\end{cor}

We provide yet another consequence of Theorems~\ref{3*k} and~\ref{n/3}: \emph{among $k$-chromatic graphs on at most $3k$ vertices, the largest choice number is attained by the complete $k$-partite graph in which every part has size $3$.} We conjecture that a similar property holds for graphs on at most $mk$ vertices for all $m$.

\begin{defn}
For $m,k\geq 2$, let $K_{m*k}$ denote the complete $k$-partite graph in which every part has size $m$.
\end{defn}

\begin{conj}
\label{m*k}
For $m,k\geq 2$, every $k$-chromatic graph $G$ on at most $mk$ vertices satisfies $\ch(G)\leq \ch(K_{m*k})$. 
\end{conj}

\begin{rem}
The main results of this thesis imply that Conjecture~\ref{m*k} is true in the case that $m\leq 3$.
\end{rem}

In actuality, we suspect that it is artificial to consider only graphs for which the number of vertices is bounded by an integer multiple of the chromatic number. For this reason, we propose the following refinement of Conjecture~\ref{m*k}.

\begin{conj}
\label{n/k}
For $n\geq k\geq 2$, there exists a graph $G_{n,k}$ such that
\begin{enumerate}[leftmargin=*,labelindent=1em,label=$\bullet$]
\item $G_{n,k}$ is a complete $k$-partite graph on $n$ vertices,
\item $\alpha\left(G_{n,k}\right)=\left\lceil\frac{n}{k}\right\rceil$, and
\item every $k$-chromatic graph $G$ on at most $n$ vertices satisfies $\ch(G)\leq \ch\left(G_{n,k}\right)$. 
\end{enumerate}
\end{conj}

\begin{rem}
Theorem~\ref{OCT} implies that Conjecture~\ref{n/k} is true when $n\leq 2k+1$. Also, by Theorems~\ref{n/3} and~\ref{1and3}, we have that Conjecture~\ref{n/k} is true in the case that $n\leq 3k$ and $n-k$ is even. 
\end{rem}

As a first step to studying Conjecture~\ref{m*k}, it is natural to try to obtain bounds on the choice number of $K_{m*k}$. Presently, for any fixed $m\geq 4$, the exact value of $\ch(K_{m*k})$ is not known for general $k$. The best known bounds for $m=4$ are given by the following result of Yang~\cite{Daqing}.

\begin{thm}[Yang~\cite{Daqing}]
\label{4*k}
$\left\lfloor\frac{3k}{2}\right\rfloor\leq \ch(K_{4*k})\leq\left\lceil\frac{7k}{4}\right\rceil$. 
\end{thm}

As a corollary of Theorem~\ref{n/3}, we obtain an improvement of the upper bound. 

\begin{cor}[Noel et al.~\cite{NWWZ}]
\label{4*k2}
$\ch(K_{4*k})\leq\left\lceil\frac{5k-1}{3}\right\rceil$. 
\end{cor}

Thus, it is known that the choice number of $K_{4*k}$ is between $\left\lfloor\frac{3k}{2}\right\rfloor$ and $\left\lceil\frac{5k-1}{3}\right\rceil$ inclusive. It is not clear whether either of these bounds hold with equality for general $k$. We speculate that the actual value of $\ch(K_{4*k})$ is usually closer to $\left\lceil\frac{5k-1}{3}\right\rceil$ than $\left\lfloor\frac{3k}{2}\right\rfloor$.

The problem of bounding $\ch(K_{m*k})$ dates back to the original paper of Erd\H{o}s et al.~\cite{ERT} on choosability, who proposed it as an approach to obtaining an estimate on the choice number of the random graph $G(n,1/2)$.\footnote{$G(n,1/2)$ is a graph on $n$ vertices in which each edge is present independently with probability $1/2$.} Using a probabilistic argument, Alon~\cite{Alonlog} determined the asymptotic behaviour of $\ch(K_{m*k})$ to within a constant factor and, as a result, showed that $\ch(G(n,1/2))=o(n)$ almost surely.

\begin{thm}[Alon~\cite{Alonlog}]
\label{logm}
There exists constants $c_1$ and $c_2$ such that
\[c_1\log(m)k\leq \ch(K_{m*k})\leq c_2\log(m)k.\]
\end{thm}

Using well-known results on the chromatic number and stability number of $G(n,1/2)$, it is not difficult to derive the following corollary from Theorem~\ref{logm}.\footnote{Specifically, $\chi(G(n,1/2))=\Theta\left(\frac{n}{\log(n)}\right)$ and $\alpha(G(n,1/2))=O(\log(n))$ almost surely.}

\begin{cor}[Alon~\cite{Alonlog}]
$\ch(G(n,1/2))=O\left(\frac{cn\log\log(n)}{\log(n)}\right)$ almost surely. 
\end{cor}

Improving on Theorem~\ref{logm}, Gazit and Krivelevich~\cite{Kriv} determined the exact asymptotics of the choice number for complete multipartite graphs whose parts may have different sizes, provided that the size of the smallest part is `not too small' compared to the size of the largest part. Their main theorem implies the following result regarding the asymptotics of the choice number of $K_{m*k}$.

\begin{thm}[Gazit and Krivelevich~\cite{Kriv}]
\label{Gazit}
$ch(K_{m*k})=(1+o(1))\left(\frac{\log(m)}{\log\left(\frac{k+1}{k}\right)}\right)$.\footnote{Since $\log\left(\frac{k+1}{k}\right) =\Theta\left(\frac{1}{k}\right)$, we see that Theorem~\ref{Gazit} implies Theorem~\ref{logm}.}
\end{thm}

Thus, the asymptotic behaviour of $\ch(K_{m*k})$ is quite well understood, despite the fact that the exact value is only known in a few cases.

\section{On-Line Ohba's Conjecture}

The on-line choice number of a graph $G$, denoted $\OL(G)$, is defined in terms a competitive game between two players, Lister and Painter, described as follows. Let $k$ be a positive integer and let $C=\{c_1,c_2,\dots\}$ be an infinite set of colours. Initially, every vertex $v$ of $G$ has $L(v):=\emptyset$. On the $i$th step of the game, where $i\geq1$, Lister chooses a non-empty set $V_i$ of vertices which are not yet coloured and, for each $v\in V_i$, replaces $L(v)$ with $L(v)\cup\{c_i\}$. Once this is done, Painter is required to choose a stable subset of $V_i$ to colour with $c_i$.\footnote{It is important to note that Painter is not allowed to modify his choice from any previous step.} If, after the $i$th step, there is a vertex $v\in V_i$ such that $v$ is not coloured with $c_i$ and $|L(v)|=k$, then Lister wins the game. On the other hand, if the game continues until every vertex is coloured, then Painter wins. We say that $G$ is \emph{on-line $k$-choosable} if the Painter has a winning strategy, and the \emph{on-line choice number} of $G$, denoted $\OL(G)$, is the minimum $k$ such that $G$ is on-line $k$-choosable.

On-line choosability was introduced independently by Schauz~\cite{MrPaint} and Zhu~\cite{Zhu}. It is clear that $\OL\geq \ch$ since, if $L$ is a list assignment for which there is no acceptable colouring, then Lister can play the game in such a way that the resulting lists are precisely $L$. 

On the other hand, the on-line choice number is also bounded above by a function of the choice number. To see this, we first require a definition. The \emph{colouring number} of a graph $G$ is defined in the following way: 
\[\col(G):=\max\left\{\delta(H)+1: H\subseteq G\right\}.\] 
Using a simple greedy colouring procedure, it is quite easy to see that every graph $G$ is $\col(G)$-choosable (see, e.g.,~\cite{JensenToft}). In fact, Zhu~\cite{Zhu} proved that every graph is also on-line $\col(G)$-choosable. Combining this with the following well known theorem of Alon~\cite{Alon}, we see that the on-line choice number is bounded above by a function of the choice number.

\begin{thm}[Alon~\cite{Alon}]
\label{colch}
There is a  function $g$ on $\mathbb{N}$ such that, for every graph $G$, we have $\col(G)\leq g(\ch(G))$.
\end{thm}

\begin{cor}[Zhu~\cite{Zhu}]
\label{OLch}
There is a function $g$ on $\mathbb{N}$ such that, for every graph $G$, we have $\OL(G)\leq g(\ch(G))$.
\end{cor}

We remark that the function $g$ in the above results is exponential. Regarding the optimal upper bound on $\OL$ in terms of $\ch$, Zhu~\cite{Zhu} asked the following questions, all of which remain unanswered: 

\begin{ques}[Zhu~\cite{Zhu}]
\label{minus}
Are there graphs for which $\OL-\ch$ is arbitrarily large?
\end{ques}

\begin{ques}[Zhu~\cite{Zhu}]
\label{over}
Are there graphs for which $\frac{\OL}{\ch}$  is arbitrarily large? 
\end{ques}

\begin{ques}[Zhu~\cite{Zhu}]
\label{poly}
Is there a polynomial function $g$ such that $\OL(G)\leq g(\ch(G))$ for every graph $G$?
\end{ques}

Regarding Question~\ref{minus}, Kim, Kwon, Liu and Zhu~\cite{online3} proved that for $k\geq3$ the complete $k$-partite graph with $1$ part of size $3$ and $k-1$ parts of size $2$ is not on-line $k$-choosable, despite the fact that it is $k$-choosable by Theorem~\ref{3222}. Therefore, for this graph, we have $\OL>\ch$.  To our knowledge, it is not known if there is a graph $G$ for which $\OL(G)>\ch(G)+1$.\footnote{This problem is mentioned by Carraher et al.~\cite{WestPaint}.} In~\cite{online2}, Kozik, Micek and Zhu suggested that the graphs $K_{3*k}$ may be natural candidates for answering Question~\ref{minus} in the positive. In the same paper, they proved that $\OL(K_{3*k})\leq\frac{3k}{2}$. Currently, it is not known whether there exists an integer $k$ such that $\OL(K_{3*k})>\ch(K_{3*k})$.

As it turns out, many of the well known upper bounds on the choice number for graphs of certain classes are also true for the on-line choice number. For example, Schauz~\cite{MrPaint} proved that every planar graph is on-line $5$-choosable, and that the line graph of every multigraph satisfies $\OL=\chi$. 

Huang et al.~\cite{online} proved that $\OL(K_{2*k})=k$ for all $k$ (c.f. Theorem~\ref{222}) by applying an on-line version of Combinatorial Nullstellensatz developed by Schauz~\cite{Schauz}.\footnote{Combinatorial Nullstellensatz proves the existence of an acceptable colouring, but does not necessarily provide a method for constructing it. For constructive proofs that $\OL(K_{2*k})$ is equal to $k$, see~\cite{online} and~\cite{WestPaint}.} Using this result as evidence, they conjectured that a slightly restricted version of Ohba's Conjecture holds for on-line choosability.  

\begin{OLOC}[Huang et al.~\cite{online}]
If $|V(G)|\leq 2\chi(G)$, then $\OL(G)=\chi(G)$. 
\end{OLOC}

Note that the result of Kim et al.~\cite{online3} cited above shows that On-Line Ohba's Conjecture cannot be extended to graphs  of order $2\chi+1$. 

Since Painter cannot change his choice from any previous step of the game, we see that Hall's Theorem cannot be directly applied to problems in on-line choosability. This suggests that the approach that we used to prove Ohba's Conjecture in Chapter~\ref{proof} cannot be applied to prove the on-line variant. 

Some special cases of On-Line Ohba's Conjecture have been proved. In~\cite{online2}, Kozik et al. proved that $\OL(G)=\chi(G)$ for graphs $G$ on at most $\chi(G)+\sqrt{\chi(G)}$ vertices (c.f. Theorem~\ref{OhbaThm}). This was improved by Carraher et al.~\cite{WestPaint} to the following.

\begin{thm}[Carraher et al.~\cite{WestPaint}]
If $|V(G)|\leq \chi(G)+2\sqrt{\chi(G)-1}$, then $\OL(G)=\chi(G)$. 
\end{thm}

Presently, however,  it is not known if there is a constant $a>1$ such that every graph on at most $a\chi$ vertices satisfies $\OL=\chi$. 

As with the early results on Ohba's Conjecture, one approach has been to verify On-Line Ohba's Conjecture for graphs of bounded stability number. Kozik et al.~\cite{online2} proved the case of graphs with stability number at most $3$. 

\begin{thm}[Kozik et al.~\cite{online2}]
If $|V(G)|\leq 2\chi(G)$ and $\alpha(G)\leq 3$, then $\OL(G)=\chi(G)$. 
\end{thm}

Another approach has been to verify On-Line Ohba's Conjecture for special classes of complete multipartite graphs. For results of this type, see~\cite{online,online3}.

\let\newpage\relax
\let\clearpage\relax
}

}
{
\newpage
\clearpage
\chapter{A Strengthening of Ohba's Conjecture}
\label{strength}

\begin{chapquote}{Paul Erd\H{o}s}
What is the purpose of Life? Proof and conjecture, and keep the SF's score low.\end{chapquote}

\doublespacing 

In this chapter, we present a proof of Theorem~\ref{n/3} from~\cite{NWWZ}. As in the proof of Ohba's Conjecture, it suffices to prove Theorem~\ref{n/3} for complete multipartite graphs. So, throughout this chapter, we fix $G$ and $L$ as follows:
\begin{enumerate}[leftmargin=*,labelindent=1em,label=$\bullet$]
\item $G$ is a complete $k$-partite graph on $n$ vertices, and
\item $L$ is a list assignment of $G$ such that $|L(v)|\geq \max\left\{k,\left\lceil\frac{n+k-1}{3}\right\rceil\right\}$ for all $v\in V(G)$.
\end{enumerate}
 We assume to the contrary that Theorem~\ref{n/3} is false and let $G$ be a minimal counterexample in the sense that, there is no acceptable colouring for $L$, but Theorem~\ref{n/3} is true for all graphs on fewer than $n$ vertices. Since Ohba's Conjecture is true, it must be the case that that $n\geq 2k+2$. Also, by Observation~\ref{<n}, we can assume that $|C_L|<n$.

In proving Theorem~\ref{n/3}, our first step is to apply the minimality assumption on $G$ to obtain several useful reductions. In contrast to the proof of Ohba's Conjecture, we are able to obtain strong restrictions on $G$ and $L$ with relatively little effort. In particular, in Section~\ref{prered}, we use the minimality assumption in a very basic way to prove that every colour of $C_L$ is available for at most $2$ vertices in each part of $G$. We then apply this result to show that the stability number of $G$ is at most $4$. 

Given that every colour of $C_L$ is available for at most $2$ vertices in each part of $G$, it is clear that an acceptable colouring for $L$ must induce a partitioning of $V(G)$ into colour classes of size at most $2$. In Section~\ref{outline}, we describe conditions under which we can use Hall's Theorem to obtain an acceptable colouring for $L$ from a pre-determined partitioning of $V(G)$ into colour classes of size at most $2$. The final step of the proof, which we provide in Section~\ref{contract}, is to show that we can partition $V(G)$ into colour classes which satisfy the conditions described in Section~\ref{outline}.

In completing this final step, we will often exploit the properties of $G$ and $L$ obtained in Section~\ref{prered}. For instance, we have already mentioned that the stability number of $G$ is at most $4$. In Section~\ref{prered}, we also show that for any part $P$ of size $2$ in $G$, the vertices of $P$ are assigned to disjoint lists. Therefore, when partitioning $V(G)$ into colour classes, we need only determine which pairs of vertices in the parts of size $3$ and $4$ should form colour classes of size $2$. 

Let us briefly mention some of the intuition which guides our choice of partitioning. Recall from Chapter~\ref{background} that it is often important to be aware of how \emph{close together} or \emph{spread apart} the lists of a given part are. In defining the partitioning of $V(G)$ into colour classes, we will insist that every pair $u$ and $v$ which forms a colour class of size $2$ has a relatively large number of available colours in common. As this suggests, the number of colour classes of size $2$ in the chosen partitioning is closely tied to the number of non-adjacent pairs in $V(G)$ which share a large number of available colours.

Given that we prefer the vertices of every non-singleton colour class to share a large number of available colours, a natural approach is to simply choose each non-singleton colour class $\{u,v\}$ so that $|L(u)\cap L(v)|$ is as large as possible. While this greedy approach does have certain advantages, it is not always necessary (and perhaps not wise) to determine every non-singleton colour class in this way. Instead, our method of partitioning $V(G)$ involves a combination of greedy and non-greedy choices. However we remark that, even when we choose to define a colour class $\{u,v\}$ in a non-greedy way, our procedure will still tend to favour pairs for which $|L(u)\cap L(v)|$ is relatively large. For a detailed description of the partitioning of $V(G)$ into colour classes, see Section~\ref{contract}.

In the final section of the chapter, we give a proof of Theorem~\ref{1and3} which provides us with a family of tight examples for Theorem~\ref{n/3}. 

\section{Preliminary Reductions}
\label{prered}

The following proposition describes a general condition under which we can apply the minimality assumption on $G$. One should compare this with the more complicated Proposition~\ref{minimal} which we used to prove Ohba's Conjecture. 

\begin{prop}
\label{min}
Suppose that $A\subseteq V(G)$ is a stable set and $c\in\cap_{v\in A}L(v)$. If $\left\lceil\frac{|V(G-A)|+\chi(G-A) - 1}{3}\right\rceil\leq \left\lceil\frac{n+k - 1}{3}\right\rceil-1$, then there is an acceptable colouring for $L$.
\end{prop}

\begin{proof}
Recall that, since Ohba's Conjecture is true, we must have $n\geq 2k+2$. This implies that $\left\lceil\frac{n+k-1}{3}\right\rceil\geq k+1$ and so 
\begin{equation}\label{k+1}|L(v)|\geq k+1\text{ for every }v\in V(G).\end{equation}

We define $L'(x):=L(x)-c$ for every vertex $x\in V(G-A)$. Then, by (\ref{k+1}) we have for $x\in V(G-A)$,
\[|L'(x)|\geq |L(x)|-1\geq k\geq \chi(G-A).\]
Now, for $x\in V(G-A)$ we have, by hypothesis,
\[|L'(x)|\geq |L(x)|-1\geq \left\lceil\frac{n+k - 1}{3}\right\rceil-1\geq \left\lceil\frac{|V(G-A)|+\chi(G-A) - 1}{3}\right\rceil.\]
Putting this together, we see that for every $x\in V(G-A)$,
\[|L'(x)|\geq \max\left\{\chi(G-A), \left\lceil\frac{|V(G-A)|+\chi(G-A) - 1}{3}\right\rceil\right\}.\]
So, by minimality of $G$, there is an acceptable colouring $f'$ for $L'$. However, we can extend $f'$ to an acceptable colouring for $L$ by colouring each vertex of $A$ with $c$. This contradiction completes the proof. 
\end{proof}

We apply Proposition~\ref{min} to prove several key reductions. 

\begin{lem}
\label{no3}
If $\{u,v,w\}$ is a stable set of $G$, then $L(u)\cap L(v)\cap L(w)=\emptyset$. 
\end{lem}

\begin{proof}
Otherwise, we set $A:=\{u,v,w\}$ and let $c$ be any element of $L(u)\cap L(v)\cap L(w)$. Then, since $|A|=3$, we have
\[\left\lceil\frac{|V(G-A)|+\chi(G-A) - 1}{3}\right\rceil\leq \left\lceil\frac{(n-3)+ k - 1}{3}\right\rceil = \left\lceil\frac{n+k-1}{3}\right\rceil-1.\]
Therefore, the result follows by Proposition~\ref{min}.
\end{proof}

\begin{lem}
\label{no2cap}
If $P=\{u,v\}$ is a part of $G$, then $L(u)\cap L(v)=\emptyset$. 
\end{lem}

\begin{proof}
Otherwise, we set $A:=P$ and let $c$ be any element of $L(u)\cap L(v)$. Then, since $|A|=2$ and $\chi(G-A)=k-1$, we have
\[\left\lceil\frac{|V(G-A)|+\chi(G-A) - 1}{3}\right\rceil\leq \left\lceil\frac{(n-2)+ (k-1) - 1}{3}\right\rceil = \left\lceil\frac{n+k-1}{3}\right\rceil-1.\]
Therefore, the result follows by Proposition~\ref{min}.
\end{proof}

During the the proof, it will sometimes be useful to consider the parity of $n$ and $k$. For this, we will apply the following lemma.

\begin{lem}
\label{3div}
$n+k-1\equiv 0\mod3$. 
\end{lem}

\begin{proof}
Suppose that $n+k-1\not\equiv 0\mod3$. First, let us show that there is a non-adjacent pair $u,v\in V(G)$ such that $L(u)\cap L(v)\neq \emptyset$. Let $P$ be the largest part of $G$. If the vertices of $P$ are assigned to disjoint lists, then since $P$ is the largest part of $G$,
\[|C_L|\geq\left|\cup_{v\in P}L(v)\right|= \sum_{v\in P}|L(v)| \geq|P|k\geq |V(G)|.\]
However, this contradicts the assumption that $|C_L|<|V(G)|$.

Therefore, there is a pair $u,v\in P$ such that $L(u)\cap L(v)\neq\emptyset$. We define $A:=\{u,v\}$ and let $c$ be any colour of $L(u)\cap L(v)$. Then, since $|A|=2$, we have
\[\left\lceil\frac{|V(G-A)|+\chi(G-A) - 1}{3}\right\rceil\leq \left\lceil\frac{(n-2)+ k - 1}{3}\right\rceil\]
However, since we are assuming $n+k-1\not\equiv 0\mod3$, we have
\[\left\lceil\frac{(n-2)+ k - 1}{3}\right\rceil\leq \left\lceil\frac{n+k-1}{3}\right\rceil-1.\]
Thus, the result follows by Proposition~\ref{min}. 
\end{proof}

Note that Lemma~\ref{3div} implies that $n-2k\equiv 1\mod 3$. Since we already know that $n> 2k+1$, we have the following: 

\begin{cor}
\label{+4}
$n\geq 2k+4$.
\end{cor}

Perhaps the most useful consequence of Lemma~\ref{no3} is that the stability number of $G$ is at most $4$, as we prove now. 

\begin{lem}
\label{alpha<=4}
$\alpha(G)\leq 4$. 
\end{lem}

\begin{proof}
Let $P$ be the largest part of $G$. By Lemma~\ref{no3}, each colour of $C_L$ is available for at most $2$ vertices of $P$. Therefore, since $|C_L|\leq n-1$, we must have
\begin{equation}\label{P5}2(n-1)\geq 2|C_L|\geq \sum_{v\in P} |L(v)|\geq |P|\left(\frac{n+k-1}{3}\right).\end{equation}
Now, if we suppose that $|P|\geq5$, then (\ref{P5}) would imply that $n>5k$. Since $P$ is the largest part of $G$ and $G$ contains precisely $k$ parts, this would imply that $|P|\geq 6$. However, if $|P|\geq 6$, then (\ref{P5}) can only be satisfied if $k\leq 0$. Therefore, it must be the case that $|P|\leq 4$. This completes the proof. 
\end{proof}

Putting this together, we see that an extremal counterexample $(G,L)$ to Theorem~\ref{n/3} must have some very special properties. To recap,
\begin{enumerate}[leftmargin=*,labelindent=1em,label=$\bullet$]
\item every part of $G$ has size at most $4$,
\item every colour $c\in C_L$ is available for at most $2$ vertices in each part of size $3$ or $4$, and
\item every colour $c\in C_L$ is available for at most $1$ vertex in each part of size $1$ or $2$. 
\end{enumerate}
We will exploit these three properties throughout the rest of the proof. In what follows, for $i\in\{1,2,3,4\}$, let $k_i$ denote the number of parts of $G$ of size $i$. 

\begin{obs}
\label{nk}
Clearly we have
\[k=k_1+k_2+k_3+k_4,\text{ and}\]
\[n=k_1+2k_2+3k_3+4k_4.\]
\end{obs}

We obtain a few more basic consequences of Lemma~\ref{3div} and Corollary~\ref{+4}.

\begin{cor}
\label{+42}
$2k_4+k_3\geq k_1+4$.
\end{cor}

\begin{proof}
By Observation~\ref{nk} and Corollary~\ref{+4}, we have
\[2k_4+k_3-k_1 = n -2k \geq 4.\]
The result follows. 
\end{proof}

\begin{cor}
\label{2mod}
$k_4+k_1-k_3\equiv 2\mod3$.
\end{cor} 

\begin{proof}
First, by Lemma~\ref{3div} we have that $2n+2k\equiv 2\mod3$. Also, we have that $2n+2k$ is equivalent to $(2n+2k) -3(2k+k_3+k_4-k_1)$ modulo $3$. Substituting for $n$ and $k$ using the equations of Observation~\ref{nk}, we have that $(2n+2k) -3(2k+k_3+k_4-k_1)$ is equal to $k_4+k_1-k_3$. The result follows. 
\end{proof}

\section{Proof Outline}
\label{outline}

In what follows, to \emph{contract} a non-adjacent a pair $u,v\in V(G)$ is to replace $u$ and $v$ with a single vertex, say $w$, whose list is defined by $L(w):=L(u)\cap L(v)$. The resulting vertex $w$ is called a \emph{contracted} vertex. The method that we use to construct an acceptable colouring for $L$ is to contract certain non-adjacent pairs in $V(G)$ and use Hall's Theorem to show that there is a system of distinct representatives for the resulting lists.\footnote{Given a collection of sets $X_1,\dots, X_n$, a \emph{system of distinct representatives} is a set of $n$ distinct elements $x_1,\dots, x_n$ such that $x_i\in X_i$ for $1\leq i\leq n$. Hall's Theorem can be restated in this language as follows: there exists a system of distinct representatives for $X_1,\dots, X_n$ if and only if $\left|\cup_{i\in S}X_i\right|\geq |S|$ for every $S\subseteq\{1,\dots,n\}$.} 

For clarity, if $P$ is a part of $G$, then we will let $P^*$ denote the set of vertices which results from completing the contraction procedure for $P$. We will be careful to contract pairs in such a way that the following properties are satisfied:
\begin{enumerate}[leftmargin=*,labelindent=1em,label= (P\arabic*)]
\item\label{size3} we contract exactly $1$ pair in $t_3\geq \left\lceil\frac{k_3}{3}\right\rceil$ parts $P$ of size $3$.
\item\label{size4} we contract at least $1$ pair in every part $P$ of size $4$. 
\end{enumerate}
In particular, at least $t_3+k_4$ pairs will be contracted, and so the number of vertices which will remain after the contraction procedure is complete is at most
\begin{equation}\label{numVaft}n-t_3-k_4.\end{equation}
Furthermore, to ensure that there is a system of distinct representatives for the resulting lists, we insist that the contraction procedure satisfies the following properties:

\begin{enumerate}[leftmargin=*,labelindent=1em,label= (P\arabic*)]
\addtocounter{enumi}{2}
\item \label{sdr} there is a system of distinct representatives for the lists of the contracted vertices.
\item \label{2from3} if $P$ is a part of size $3$ and $x,y\in P^*$, then $|L(x)\cup L(y)|\geq k+\left\lceil\frac{k_3}{3}\right\rceil +k_4$. 
\item\label{nomerge} if $P$ is a part of size $3$ and $P^*=P$, then for any $x,y\in P^*$ we have $|L(x)\cup L(y)|\geq k+k_3+k_4$.
\item \label{Z} there is a set $Z_3$ of exactly $\left\lfloor\frac{2k_3}{3}\right\rfloor$ parts $P$ of size $3$ such that if $P\in Z_3$ and $x,y\in P^*$, then $|L(x)\cup L(y)|\geq k+t_3+k_4$.
\item \label{3verts} for any part $P$ if $x,y,z\in P^*$, then $|L(x)\cup L(y)\cup L(z)|\geq n-t_3-k_4$.
\item \label{2merged} there is a set $Z_4$ of at least $\frac{k_4+k_1-k_3+1}{3}$ parts $P$ of size $4$ such that if $P\in Z_4$ and $x,y\in P^*$, then $|L(x)\cup L(y)|\geq k+k_4$.\footnote{Note that, by Corollary~\ref{2mod}, we have that $\frac{k_4+k_1-k_3+1}{3}$ is an integer which, by Corollary~\ref{+42}, is less than $k_4$.}
\end{enumerate}

To complete the proof of Theorem~\ref{n/3}, it suffices to prove the following two lemmas.

\begin{lem}
\label{merge}
It is possible to contract a collection of non-adjacent pairs in $V(G)$ such a way that properties \ref{size3} through \ref{2merged} are satisfied.
\end{lem}

\begin{lem}
\label{matchLem}
If we have contracted a collection of non-adjacent pairs in $V(G)$ in such a way that properties \ref{size3} through \ref{2merged} are satisfied, then there is a system of distinct representatives for the resulting lists.
\end{lem}

We leave the description of the contraction procedure and the proof of Lemma~\ref{merge} to the next section. Let us now prove Lemma~\ref{matchLem}.

\begin{proof}[Proof of Lemma~\ref{matchLem}]
Suppose to the contrary that there does not exist a system of distinct representatives for the lists after contracting a collection of non-adjacent pairs in $V(G)$ in such a way that properties \ref{size3} through \ref{2merged} are satisfied. Then, after contracting, there is a set $S$ of vertices such that $\left|\cup_{x\in S}L(x)\right|<|S|$ by Hall's Theorem. If $S$ contains $3$ vertices, say $x,y,z$, from some part $P^*$, then by \ref{3verts} we have
\[|S|>|L(x)\cup L(y)\cup L(z)| \geq n-t_3-k_4\]
which contradicts (\ref{numVaft}). Thus, $S$ contains at most $2$ vertices from each part for a total of at most $2k$ vertices. 

Suppose that $P=\{u,v\}$ is a part and that $P^*=P$. Then by Lemma~\ref{no2cap} we have that $L(u)\cap L(v)=\emptyset$. Therefore, if $u,v\in S$, then
\[|S|>|L(u)\cup L(v)|=|L(u)|+|L(v)| \geq2k.\]
However, this contradicts the bound from the previous paragraph. Therefore, $S$ contains at most $1$ vertex from each part $P^*$ such that $|P|=2$. In total, we have $|S|\leq k+k_3+k_4$.

Suppose now that there is a part $P$ of size $3$ such that $P^*=P$ and $S$ contains $2$ vertices $x,y$ of $P^*$. By \ref{nomerge}, this implies that
\[|S|>|L(x)\cup L(y)|\geq k+k_3+k_4\]
which is a contradiction. Therefore, $S$ contains at most one vertex from every such part. By \ref{size3}, we have $|S|\leq k+t_3+k_4$.

Next, suppose that, for some $P\in Z_3$, $S$ contains two vertices $x$ and $y$ of $P^*$. Then, by \ref{Z}, we would have
\[|S|>|L(x)\cup L(y)|\geq k+t_3+k_4\]
which is, again, a contradiction. Thus, $S$ contains at most $1$ vertex from each part $P^*$ such that $P\in Z_3$. Since $|Z_3|=\left\lfloor\frac{2k_3}{3}\right\rfloor$, this implies that 
\begin{equation}\label{ceil}|S|\leq (k+k_3+k_4) - \left\lfloor\frac{2k_3}{3}\right\rfloor = k+\left\lceil\frac{k_3}{3}\right\rceil+k_4.\end{equation} 

Now, if $S$ contains $2$ vertices $x$ and $y$ from \emph{any} part $P^*$  such that $|P|=3$, then by \ref{2from3} we would have
\[|S|>|L(x)\cup L(y)|\geq k+\left\lceil\frac{k_3}{3}\right\rceil+k_4.\]
This contradicts (\ref{ceil}), and therefore $|S|\leq k+k_4$. 

For $P\in Z_4$, if $S$ contains $2$ vertices $x$ and $y$ of $P^*$, then by \ref{2merged} we would have 
\[|S|>|L(x)\cup L(y)|\geq k+k_4,\]
which is a contradiction. Thus, $S$ contains at most $1$ vertex from each part $P^*$ such that $P\in Z_4$. By \ref{2merged}, we have
\[|S|\leq (k+k_4)-|Z_4|= (k+k_4)-\left(\frac{k_4+k_1-k_3+1}{3}\right) = \frac{2k_1+3k_2+4k_3+5k_4-1}{3}\]
\[= \frac{n+k-1}{3}\]
However, since $|L(x)|\geq \frac{n+k-1}{3}$ for every vertex $x$ that is not contracted, this implies that $S$ contains only contracted vertices. This contradicts \ref{sdr} and completes the proof of Lemma~\ref{matchLem}. 
\end{proof}

To prove Theorem~\ref{n/3}, all that remains is to describe the contraction procedure and to show that it satisfies properties \ref{size3} through \ref{2merged}. 

\section{The Contraction Procedure}
\label{contract}

As we mentioned at the beginning of this chapter, it will be important to keep track of how \emph{close together} or \emph{spread apart} the lists of a given part are. The following definition allows us to do this in a crude, but useful, way. 

\begin{defn}
Given a part $P$ of $G$, define
\[\ell(P):=\max\{|L(u)\cap L(v)|: u,v\in P\}.\]
\end{defn}

We use Lemma~\ref{no3} to obtain a straightforward bound on $\ell$. 

\begin{lem}
\label{geqk}
If $S\subseteq V(G)$ is a stable set such that $|S|= 3$, then
\[\sum_{u,v\in S}|L(u)\cap L(v)|\geq k.\]
\end{lem}

\begin{proof}
By the inclusion-exclusion principle, Lemma~\ref{no3}, and the fact that $\left|\cup_{v\in S}L(v)\right| \leq |C_L|<n$, we have
\[\sum_{u,v\in S}|L(u)\cap L(v)|= \sum_{v\in S}|L(v)| - \left|\cup_{v\in S}L(v)\right|\]
\[\geq 3\left(\frac{n+k-1}{3}\right)-(n-1) =k\]
as desired.
\end{proof}

\begin{cor}
\label{k/3}
For every part $P$ such that $|P|\geq3$, we have $\ell(P)\geq\left\lceil\frac{k}{3}\right\rceil$.
\end{cor}

\begin{proof}
Let $S\subseteq P$ so that $|S|=3$. By Lemma~\ref{geqk} we have $\sum_{u,v\in S}|L(u)\cap L(v)|\geq k$. Therefore, by the pigeonhole principle there must be a pair $u,v\in S$ such that 
\[|L(u)\cap L(v)|\geq \left\lceil \frac{k}{\binom{3}{2}}\right\rceil = \left\lceil \frac{k}{3}\right\rceil\]
as desired. 
\end{proof}

A natural approach to proving Lemma~\ref{merge} is to start by contracting pairs $u,v$ in certain parts $P$, where $u$ and $v$ are chosen to satisfy $|L(u)\cap L(v)|=\ell(P)$. Certainly, it is often advantageous to contract pairs in such a way that the resulting lists are as large as possible. However, to ensure that there is a system of distinct representatives for the lists of the contracted vertices, we will usually aim for a more relaxed condition on the size of $|L(u)\cap L(v)|$. This condition is described by the following definition. 

\begin{defn}
Let $P$ be a part of size $3$ or $4$ and let $u,v\in P$. We say that $(u,v)$ is a \emph{good} pair for $P$ if either
\begin{enumerate}[leftmargin=*,labelindent=1em,label=$\bullet$]
\item $|P|=3$ and $|L(u)\cap L(v)|\geq\left\lceil\frac{k_1+k_4+1}{3}\right\rceil$, or
\item $|P|=4$ and for $w,z\in P-\{u,v\}$ we have $|L(u)\cap L(v)|\geq |L(w)\cap L(z)|$. 
\end{enumerate}
\end{defn}

We insist that the contraction procedure satisfies the following property:
\begin{enumerate}[leftmargin=*,labelindent=1em,label= (Q\arabic*)]
\item \label{good} if $P$ is a part such that we contract exactly one pair $u,v\in P$, then $(u,v)$ is a good pair for $P$. 
\end{enumerate}
We obtain some consequences of \ref{good}.

\begin{lem}
\label{goodisgood3}
If the contraction procedure satisfies \ref{good} and \ref{nomerge}, then it also satisfies \ref{2from3}. 
\end{lem}

\begin{proof}
Suppose that the contraction procedure satisfies \ref{good} and \ref{nomerge}. Let $P$ be any part of size $3$ and let $x,y\in P^*$ be arbitrary. If $P^*$ does not contain a contracted vertex, then by \ref{nomerge} we have
\[|L(x)\cap L(y)|\geq k+k_3+k_4\geq k+\left\lceil\frac{k_3}{3}\right\rceil +k_4.\]
Therefore, \ref{2from3} is satisfied in this case.

On the other hand, suppose that $x$ is the unique contracted vertex of $P^*$. Then, by Lemma~\ref{no3} we have $L(x)\cap L(y)=\emptyset$ and so \ref{good} implies that
\[|L(x)\cup L(y)|=|L(x)|+|L(y)|\geq \left\lceil\frac{k_1+k_4+1}{3}\right\rceil + \frac{n+k-1}{3}\geq k+\frac{k_3}{3}+k_4\]
and so \ref{2from3} holds.
\end{proof}

\begin{lem}
\label{goodisgood4}
If the contraction procedure satisfies \ref{good}, \ref{size3} and \ref{size4}, then it also satisfies \ref{3verts} for every part $P$ of size $4$.
\end{lem}

\begin{proof}
Suppose that the contraction procedure satisfies \ref{good}, \ref{size3} and \ref{size4}. Let $P$ be any part of size $4$ and let $x,y,z\in P^*$. Then, since \ref{size4} is satisfied, we can assume, without loss of generality, that $x$ is contracted and $y$ and $z$ are not. Clearly, $x$ is the unique contracted vertex in $P^*$. Since \ref{good} is satisfied, we have that $|L(x)|\geq |L(y)\cap L(z)|$. Also, by Lemma~\ref{no3} we must have $L(x)\cap L(y)=L(x)\cap L(z)=\emptyset$. Therefore, since we are assuming that \ref{size3} is true we have $t_3\geq\left\lceil\frac{k_3}{3}\right\rceil$, and so
\[|L(x)\cup L(y)\cup L(z)| = |L(x)|+|L(y)|+|L(z)|-|L(y)\cap L(z)|\]
\[\geq |L(y)|+|L(z)| \geq 2\left(\frac{n+k-1}{3}\right)\geq n-\frac{k_3}{3} -k_4 -\frac{2}{3}\]
\[\geq n-t_3-k_4-\frac{2}{3}\]
which implies that $|L(x)\cup L(y)\cup L(z)|\geq n-t_3-k_4$. The result follows. 
\end{proof}

Using Corollary~\ref{k/3}, we obtain a simple condition under which $(u,v)$ is guaranteed to be a good pair for a part $P$.

\begin{lem}
\label{ellgood}
If $|P|\geq 3$ and $u,v\in P$ such that $|L(u)\cap L(v)|=\ell(P)$, then $(u,v)$ is a good pair for $P$. 
\end{lem}

\begin{proof}
If $|P|=4$, then the result follows easily from the definition of a good pair for $P$. So, suppose that $|P|=3$. In this case, we obviously have $k_3\geq 1$ and so by Corollary~\ref{k/3}
\[|L(u)\cap L(v)|= \ell(P)\geq\left\lceil\frac{k}{3}\right\rceil =\left\lceil\frac{k_1+k_2+k_3+k_4}{3}\right\rceil\geq \left\lceil\frac{k_1+k_4+1}{3}\right\rceil.\]
Therefore, $(u,v)$ is a good pair for $P$. The result follows. 
\end{proof}

Finally, we describe the contraction procedure. We consider parts of size $3$ and $4$ separately. 

\subsection{Parts of Size 3}

Let $Z_3$ be an arbitrary set of $\left\lfloor\frac{2k_3}{3}\right\rfloor$ parts $P$ of size $3$. We order the elements of $Z_3$ so that $\ell(P)$ is decreasing. Let $t_3$ denote the maximum integer $i$ such that every part $P$ in the set $Z_3'$ of the first $i-\left\lceil\frac{k_3}{3}\right\rceil$ parts in $Z_3$ satisfies $\ell(P)\geq\left\lceil\frac{k+i-1}{3}\right\rceil$. Note that, since $|Z_3'|=t_3-\left\lceil\frac{k_3}{3}\right\rceil$ and $0\leq |Z_3'| \leq |Z_3|=\left\lfloor\frac{2k_3}{3}\right\rfloor$, we must have 
\begin{equation}\label{t3}\left\lceil\frac{k_3}{3}\right\rceil\leq t_3\leq \left\lceil\frac{k_3}{3}\right\rceil + \left\lfloor\frac{2k_3}{3}\right\rfloor=k_3.\end{equation}
We contract pairs in the parts of size $3$ in the following way:

\begin{itemize}
\item for each part $P\in Z_3'$, we contract a pair $u,v\in P$ such that $|L(u)\cap L(v)|=\ell(P)$.
\item for each part $P\in Z_3-Z_3'$, we do not contract any pair in $P$.
\item at a later stage of the procedure, for each part $P$ of size $3$ which is not contained in $Z_3$, we will contract a good pair in $P$. 
\end{itemize}

Therefore, the number of parts which were originally of size $3$ and will contain a contracted vertex is precisely
\[|Z_3'| + \left(k_3-|Z_3|\right) = \left(t_3-\left\lceil\frac{k_3}{3}\right\rceil\right) + \left(k_3 -\left\lfloor\frac{2k_3}{3}\right\rfloor\right) = t_3.\]
Combining this with (\ref{t3}), we have that \ref{size3} is satisfied. Also, by Lemma~\ref{ellgood}, we have that \ref{good} is satisfied for every part $P$ of size $3$.  

To close this subsection, we prove that the contraction procedure satisfies \ref{nomerge} and \ref{Z}, and that it satisfies \ref{3verts} for every part $P$ of size $3$. By Lemma~\ref{goodisgood3} and the fact that \ref{good} holds for every part $P$ of size $3$, it will follow that \ref{2from3} holds as well. 

\begin{lem}
\label{satnomerge}
The contraction procedure satisfies \ref{nomerge}.
\end{lem}

\begin{proof}
Let $P$ be a part of size $3$ such that $P^*=P$. By definition of the contraction procedure, we have $P\in Z_3-Z_3'$. Let $x,y\in P^*$ be arbitrary. Then, by definition of $Z_3'$, we have that $|L(x)\cap L(y)|\leq\left\lceil\frac{k+t_3}{3}\right\rceil-1$. Therefore, since $k_3\geq t_3$, 
\[|L(x)\cup L(y)| = |L(x)|+|L(y)|-|L(x)\cap L(y)|\]
\[ = 2\left(\frac{n+k-1}{3}\right) - \left(\left\lceil\frac{k+t_3}{3}\right\rceil -1\right)\geq \frac{2n+k-t_3-1}{3} \geq k + \frac{4k_3-t_3-1}{3} +2k_4\]
\[\geq k +k_3+k_4 -\frac{1}{3}.\]
So, we have $|L(x)\cup L(y)|\geq k +k_3+k_4$. The result follows.
\end{proof}

\begin{lem}
\label{satZ}
The contraction procedure satisfies \ref{Z}.
\end{lem}

\begin{proof}
Let $P\in Z_3$ and $x,y\in P^*$ be arbitrary. If $P\in Z_3-Z_3'$, then the result follows from Lemma~\ref{satnomerge}. Thus, we assume that $P\in Z_3'$. Without loss of generality, we can assume that $x$ is a contracted vertex and $y$ is not. By Lemma~\ref{no3}, this implies that $L(x)\cap L(y)=\emptyset$. Now, by definition of $Z_3'$, we have
\[|L(x)\cup L(y)|=|L(x)|+|L(y)|\geq \left\lceil\frac{k+t_3-1}{3}\right\rceil +\frac{n+k-1}{3}\]
\[\geq k+\frac{2k_3+t_3}{3} + k_4 -\frac{2}{3}\]
which, since $t_3\leq k_3$, implies that $|L(x)\cup L(y)|\geq k+t_3+k_4$. The result follows. 
\end{proof}

\begin{lem}
\label{sat3verts3}
The contraction procedure satisfies \ref{3verts} for every part $P$ of size $3$.
\end{lem}

\begin{proof}
Let $P$ be a part of size $3$ such that $P^*=P$. By definition of the contraction procedure, we have $P\in Z_3-Z_3'$. Thus, by definition of $Z_3'$, every pair $x,y\in P^*$ satisfies $|L(x)\cap L(y)|\leq \left\lceil\frac{k+t_3}{3}\right\rceil -1$. It follows that
\[\left|\cup_{x\in P^*}L(v)\right| = \sum_{x\in P^*}|L(x)| - \sum_{x,y\in P^*}|L(x)\cap L(y)|\]
\[\geq 3\left(\frac{n+k-1}{3}\right)-3\left(\left\lceil\frac{k+t_3}{3}\right\rceil -1\right) \geq (n+k-1)-(k+t_3-1)\geq n-t_3-k_4.\]
The result follows. 
\end{proof}

Let us summarize the results of this subsection.

\begin{note}
\label{summary}
We have shown that the contraction procedure satisfies \ref{size3}, \ref{2from3}, \ref{nomerge} and \ref{Z} and that it satisfies \ref{3verts} and \ref{good} for every part of size $3$.
\end{note}

\subsection{Parts of Size 4}

Let $Z_4$ be an arbitrary set of $\max\left\{0,\frac{k_4+k_1-k_3+1}{3}\right\}$ parts of size $4$. We contract pairs in the parts of size $4$ as follows:
\begin{itemize}
\item for each part $P\in Z_4$, we first contract a pair $u,v\in P$ such that $|L(u)\cap L(v)|=\ell(P)$. Afterwards, let $w$ and $z$ denote the two vertices of $P-\{u,v\}$. If $|L(w)\cap L(z)|\geq \frac{k_1+3k_2+5k_3+4k_4+1}{3}$, then we also contract $w$ and $z$.\footnote{It can be easily deduced from Corollary~\ref{2mod} that $\frac{k_1+3k_2+5k_3+4k_4+1}{3}$ is an integer.}
\item at a later stage of the procedure, for each part $P$ of size $4$ which is not contained in $Z_4$, we will contract exactly one good pair in $P$.
\end{itemize}

\begin{note}
\label{Z4big}
By Corollary~\ref{k/3} and the definition of the contraction procedure, for every contracted vertex $x$ in a part $P^*$ such that $P\in Z_4$ we have $|L(x)|\geq\left\lceil\frac{k}{3}\right\rceil$. 
\end{note}

By Lemma~\ref{ellgood} and the definiton of the contraction procedure, we will have contracted a good pair in every part $P$ such that $|P|=4$. This implies that both \ref{size4} and \ref{good} are satisfied. Therefore, by Lemma~\ref{goodisgood4} and Lemma~\ref{sat3verts3}, we have that the contraction procedure satisfies \ref{3verts} in general. 

We prove that the contraction procedure satisfies \ref{2merged}. Once this is done, in order to complete the proof of Theorem~\ref{n/3}, all that remains is to describe the contraction procedure for parts of size $3$ and $4$ not contained in $Z_3\cup Z_4$, and to show that it can be done in such a way that \ref{sdr} is satisfied.

\begin{lem}
The contraction procedure satisfies \ref{2merged}.
\end{lem}

\begin{proof}
By definition, we have that $|Z_4|\geq\frac{k_4+k_1-k_3+1}{3}$. Let $P\in Z_4$ and $x,y\in P^*$ be arbitrary. Our goal is to show that $|L(x)\cup L(y)|\geq k+k_4$.

First, suppose that neither $x$ nor $y$ is a contracted vertex. Then, by the definition of the contraction procedure, we must have that $|L(x)\cap L(y)|\leq \frac{k_1+3k_2+5k_3+4k_4+1}{3}-1$. So, 
\[|L(x)\cup L(y)| = |L(x)|+|L(y)| - |L(x)\cap L(y)|\]
\[ \geq 2\left(\frac{n+k-1}{3}\right) -\left(\frac{k_1+3k_2+5k_3+4k_4+1}{3}-1\right)= k+k_4.\]
Therefore, we have $|L(x)\cup L(y)|\geq k+k_4$. 

Now, suppose that $x$ is contracted any $y$ is not. Then it must be the case that $x$ is the result of contracting a pair $u,v\in P$ such that $|L(u)\cap L(v)|=\ell(P)$. Thus, by Corollary~\ref{k/3}, we have that $|L(x)|\geq\left\lceil\frac{k}{3}\right\rceil$. Also, by Lemma~\ref{no3}, we have $L(x)\cap L(y)=\emptyset$. Therefore,
\[|L(x)\cup L(y)|=|L(x)|+|L(y)| \geq \left\lceil\frac{k}{3}\right\rceil + \frac{n+k-1}{3}\]
\[\geq k+k_4 -\frac{1}{3}\]
which implies that $|L(x)\cup L(y)|\geq k+k_4$. 

Finally, suppose that both $x$ and $y$ are contracted. Then, without loss of generality, $x$ is the result of contracting a pair $u,v\in P$ such that $|L(u)\cap L(v)|=\ell(P)$ and $y$ is the result of contracting a pair $w,z\in P$ such that $|L(w)\cap L(z)|\geq \frac{k_1+3k_2+5k_3+4k_4+1}{3}$. However, by definition of $\ell$, this implies that both $|L(x)|$ and $|L(y)|$ are at least $\frac{k_1+3k_2+5k_3+4k_4+1}{3}$. Also, by Lemma~\ref{no3}, we have $L(x)\cap L(y)=\emptyset$. It follows that
\[|L(x)\cup L(y)|=|L(x)|+|L(y)|\geq 2\left(\frac{k_1+3k_2+5k_3+4k_4+1}{3}\right)\]
\[\geq (k+k_4) +\frac{7k_3-k_1+2k_4}{3}.\]
By Corollary~\ref{+42} we have that the second term is positive, and so $|L(x)\cup L(y)|\geq k+k_4$ in this case as well. The result follows. 
\end{proof}

\subsection{Final Arguments}

The last step of the proof is to show that it is possible to contract a good pair in every part of size $3$ or $4$ not in $Z_3\cup Z_4$ in such a way that there is a system of distinct representatives for the lists of the contracted vertices; i.e. to ensure that the contraction procedure satisfies \ref{sdr} and \ref{good}. To this end, we let $Y$ be the set of all parts of size $3$ or $4$ not contained in $Z_3\cup Z_4$ and, for each part $P\in Y$, we create an auxiliary vertex $v_P$ corresponding to $P$. Let $\Gamma_P$ denote the set of good pairs for $P$, and define the list for $v_P$ as follows:
\begin{equation}\label{goodvP}L(v_P):=\bigcup_{(u,v)\in\Gamma_P}\left(L(u)\cap L(v)\right).\end{equation}

\begin{defn} 
We let $X$ be the set consisting of $v_P$ for every $P\in Y$, and every vertex which has been contracted at an earlier stage of the procedure.
\end{defn}

Our goal is to show that there is a system of distinct representatives for the lists of vertices in $X$. If so, then for each part $P\in Y$, we contract a good pair $(u,v)$ for $P$ such that $L(u)\cap L(v)$ contains the representative for $L(v_P)$. By construction, there will be a system of distinct representatives for the lists of the contracted vertices. Before moving on, let us establish some bounds on $|L(v_P)|$.

\begin{lem}
\label{vP}
If $P$ is a part of size $3$ not contained in $Z_3$, then 
\[|L(v_P)|\geq k_3+\left\lceil\frac{k_1+k_4}{3}\right\rceil.\]
\end{lem}

\begin{proof}
By Lemma~\ref{geqk}, we have that $\sum_{u,v\in P}|L(u)\cap L(v)|\geq k$. Also, by Lemma~\ref{ellgood}, there are at most $2$ pairs in $P$ which are not good for $P$. By definition, every such pair must have at most $\left\lceil\frac{k_1+k_4+1}{3}\right\rceil -1$ colours in common. Thus, by Lemma~\ref{no3},
\[|L(v_P)| = \sum_{(u,v)\in \Gamma_P}|L(u)\cap L(v)| = \sum_{u,v\in P}|L(u)\cap L(v)|-\sum_{(u,v)\notin \Gamma_P}|L(u)\cap L(v)|\] 
\[\geq k-2\left(\left\lceil\frac{k_1+k_4+1}{3}\right\rceil -1\right) > k_3 +\frac{k_1+k_4}{3}.\]
The result follows. 
\end{proof}

\begin{lem}
\label{vQ}
If $P$ is a part of size $4$ not contained in $Z_4$, then 
\[|L(v_P)|\geq k_3+k_4.\]
\end{lem}

\begin{proof}
By the definition of a good pair for $P$ and (\ref{goodvP}), we have that
\begin{equation}\label{half}|L(v_P)|> \frac{1}{2}\sum_{u,v\in P}\left|L(u)\cap L(v)\right|\end{equation}
and so it suffices to obtain a lower bound on the latter. By Lemma~\ref{no3}, the inclusion-exclusion principle and the fact that $|C_L|<n$, we have
\[\sum_{u,v\in P}\left|L(u)\cap L(v)\right| = \sum_{u\in P}|L(u)| - \left|\cup_{u\in P}L(u)\right|\]
\[\geq 4\left(\frac{n+k-1}{3}\right) - (n-1) =\frac{n+4k-1}{3}.\]
Therefore, by (\ref{half}), we have
\[|L(v_P)|\geq \frac{n+4k-1}{6} \geq \frac{7k_3+8k_4-1}{6}.\]
By Corollary~\ref{+42} we have $k_3+2k_4>1$ and so the result follows.
\end{proof}

To complete the proof of Theorem~\ref{n/3}, we need only establish the following lemma.

\begin{lem}
There is a system of distinct representatives for the lists of vertices in $X$.
\end{lem}

\begin{proof}
Otherwise, by Hall's Theorem, there must be a set $S\subseteq X$ such that $\left|\cup_{v\in S}L(v)\right|<|S|$. As in the proof of Lemma~\ref{matchLem}, our goal is to obtain stronger and stronger bounds on the size of $S$ until we reach a contradiction.

For each part $P$ of size $3$ there is at most one vertex in $X$ corresponding to $P$ (i.e. if $P\notin Z_3$, then $X$ contains $v_P$ and if $P\in Z_3'$, then $X$ contains the contracted vertex in $P^*$), and for every part $P'$ of size $4$ there are at most $2$ vertices in $X$ corresponding to $P'$. This implies that
\begin{equation}\label{3+24}|S|\leq |X|\leq k_3+2k_4.\end{equation}
However, if $S$ contains two vertices $x,y$ from a part $P^*$ such that $|P|=4$, then by \ref{2merged} we would have that 
\[|S|>|L(x)\cup L(y)|\geq k+k_4 \geq k_3+2k_4\]
which contradicts (\ref{3+24}). Therefore, $S$ contains at most one vertex corresponding to each part of size $3$ or $4$, and so
\begin{equation}\label{3+4}|S|\leq k_3+k_4.\end{equation}

Suppose next that $S$ contains a vertex $v_P$ where $P$ is a part of size $4$ not contained in $Z_4$. By Lemma~\ref{vQ}, this would imply that 
\[|S|>|L(v_P)|\geq k_3+k_4.\]
However, this contradicts (\ref{3+4}). Therefore, $S$ cannot contain any vertex corresponding to a part $P$ of size $4$ not in $Z_4$. Thus, in total, we have 
\begin{equation}\label{3+Z4}|S|\leq k_3+|Z_4|= k_3 + \max\left\{0,\frac{k_4+k_1-k_3+1}{3}\right\}.\end{equation}

Now, suppose that $S$ contains a vertex $v_P$ where $P$ is a part of size $3$ not contained in $Z_3$ (in particular, this implies that $k_3\geq1$). By Lemma~\ref{vP}, we have
\[|S|>|L(v_P)|\geq  k_3 + \left\lceil\frac{k_4+k_1}{3}\right\rceil\]
which contradicts (\ref{3+Z4}). Therefore, $S$ cannot contain a vertex $v_P$ corresponding to a part $P\notin Z_3\cup Z_4$. It follows that
\begin{equation}\label{Z3'+Z4}|S|\leq |Z_3'|+|Z_4| = \left(t_3-\left\lceil\frac{k_3}{3}\right\rceil\right) + \max\left\{0,\frac{k_4+k_1-k_3+1}{3}\right\}.\end{equation}

Next, we suppose that $S$ contains a contracted vertex $x$ from a part $P^*$ such that $P\in Z_3'$. Then by definition of $Z_3'$ it must be the case that
\begin{equation}\label{oneguy3}|S|>|L(x)|\geq \left\lceil\frac{k+t_3-1}{3}\right\rceil.\end{equation}
Combining (\ref{Z3'+Z4}) and (\ref{oneguy3}), we obtain
\[\left(t_3-\left\lceil\frac{k_3}{3}\right\rceil\right) + \max\left\{0,\frac{k_4+k_1-k_3+1}{3}\right\} \geq |S|\geq \frac{k+t_3+2}{3}\]
which, in any case, contradicts the fact that $t_3\leq k_3$. Therefore, $S$ can only contain contracted vertices from parts $P^*$ such that $P\in Z_4$, and so
\begin{equation}\label{Z4}|S|\leq |Z_4| = \max\left\{0,\frac{k_4+k_1-k_3+1}{3}\right\}.\end{equation}

Finally, let $x\in S$ be arbitrary. By Note~\ref{Z4big}, we have that
\[|S|>|L(x)|\geq\left\lceil\frac{k}{3}\right\rceil \geq|Z_4|\]
which contradicts (\ref{Z4}). This completes the proof. 
\end{proof}

\section{Tight Examples}
\label{3and1}

We close this chapter with a proof of Theorem~\ref{1and3}. The construction that we exhibit is very similar to the one that was used to prove Proposition~\ref{3331}. We remark that the lower bound of Theorem~\ref{4*k} can also be proved by applying a similar principle.

\begin{proof}[Proof of Theorem~\ref{1and3}]
The upper bound is implied by Theorem~\ref{n/3}, and so it suffices to prove the lower bound. Also, if $n\leq 2k+1$, then lower bound is trivial, so we assume that $\left\lceil\frac{n+k-1}{3}\right\rceil > k$. 

Define $s:=\frac{k_1+2k_3-1}{3}$ and let $X_1,X_2$ and $X_3$ be disjoint sets of colours so that for $i,j\in\{1,2,3\}$ we have
\[|X_i\cup X_j|\geq \lfloor 2s\rfloor = \left\lceil\frac{n+k-1}{3}\right\rceil -1\]
and
\[|X_1\cup X_2\cup X_3| \leq k_1+2k_3-1.\]
We assign the lists $X_1\cup X_2$, $X_1\cup X_3$ and $X_2\cup X_3$ to the vertices of the parts of size $3$ and $X_1\cup X_2\cup X_3$ to the vertices of the parts of size $1$. If there is an acceptable colouring $f$, then it must use at least $2$ colours on each part of size $3$ and $1$ colour on each part of size $1$ for a total of at least $k_1+2k_3$ colours. However, since the total number of colours is $|X_1\cup X_2\cup X_3|<k_1+2k_3$, this is a contradiction. 
\end{proof}

\let\newpage\relax
\let\clearpage\relax

}

{
\newpage
\clearpage
\chapter{Conclusion}
\label{concl}

\begin{chapquote}{Ana\"{i}s Nin}
The possession of knowledge does not kill the sense of wonder and mystery. There is always more mystery.
 \end{chapquote}

\doublespacing 

In previous chapters, we have discussed several open problems for future study. For convenience, we conclude the thesis by repeating a few of them.

\begin{customconj}{\ref{only}}
If $G$ is a complete $k$-partite graph on $2k+2$ vertices such that $\ch(G)>k$, then either
\begin{enumerate}[leftmargin=*,labelindent=1em,label=$\bullet$]
\item every part of $G$ has size $3$ or $1$, or
\item $k$ is even and every part of $G$ has size $4$ or $2$.
\end{enumerate}
\end{customconj}

\begin{customconj}{\ref{m*k}}
For $m,k\geq2$, every $k$-chromatic graph $G$ on at most $mk$ vertices satisfies $\ch(G)\leq \ch(K_{m*k})$. 
\end{customconj}

\begin{customconj}{\ref{n/k}}
For $n\geq k\geq 2$, there exists a graph $G_{n,k}$ such that
\begin{enumerate}[leftmargin=*,labelindent=1em,label=$\bullet$]
\item $G_{n,k}$ is a complete $k$-partite graph on $n$ vertices,
\item $\alpha\left(G_{n,k}\right)=\left\lceil\frac{n}{k}\right\rceil$, and
\item every $k$-chromatic graph $G$ on at most $n$ vertices satisfies $\ch(G)\leq \ch\left(G_{n,k}\right)$. 
\end{enumerate}
\end{customconj}

\begin{prob}
Determine the choice number of $K_{4*k}$ for general $k$. 
\end{prob}

\begin{customques}{\ref{minus}}[Zhu~\cite{Zhu}]
Are there graphs for which $\OL-\ch$ is arbitrarily large?
\end{customques}

\begin{customques}{\ref{over}}[Zhu~\cite{Zhu}]
Are there graphs for which $\frac{\OL}{\ch}$  is arbitrarily large? 
\end{customques}

\begin{customques}{\ref{poly}}[Zhu~\cite{Zhu}]
Is there a polynomial function $g$ such that $\OL(G)\leq g(\ch(G))$ for every graph $G$?
\end{customques}

\begin{ques}[Carraher et al.~\cite{WestPaint}]
Is there a graph $G$ such that $\OL(G)> \ch(G)+1$?
\end{ques}

\begin{prob}[Kozik et al.~\cite{online2}]
Determine the on-line choice number of $K_{3*k}$ for general $k$. 
\end{prob}

\begin{OLOC}[Huang et al.~\cite{online}]
If $|V(G)|\leq 2\chi(G)$, then $\OL(G)=\chi(G)$. 
\end{OLOC}

\let\newpage\relax
\let\clearpage\relax

}


{
\newpage
\clearpage
\addcontentsline{toc}{chapter}{Glossary of Graph Theoretic Terminology}
\chapter*{Glossary of Graph Theoretic Terminology}

\onehalfspacing

\section*{Basic Graph Theory}

In what follows, $G$ and $H$ are graphs, $u$ and $v$ are vertices of $G$, $S$ is a set of vertices of $G$, and $M$ is a set of edges of $G$.

\begin{center}
\begin{longtable}{l|l|l}
\hline{\bf Terminology }						& {\bf Notation}	& {\bf Definition/Meaning } \\ \hline\hline
A \emph{graph} $G$									& 								& A collection of points, called \emph{vertices}, some pairs\\
																		&									& of which are joined by lines, called \emph{edges}.\\\hline
The \emph{vertex set} of $G$				& $V(G)$					& The set of vertices of $G$.\\\hline
The \emph{edge set} of $G$ 					& $E(G)$					& The set of edges of $G$.\\\hline
The \emph{order} of $G$							& $|V(G)|$				& The number of vertices of $G$.\\\hline
The \emph{size} of $G$							& $|E(G)|$				& The number of edges of $G$.\\\hline
$u$ is \emph{adjacent} to $v$ 			& $u\sim v$				& $u$ and $v$ are joined by an edge.\\\hline
The \emph{neighbourhood} of $u$			& $N_G(u)$				&	The set of vertices $v$ such that $u\sim v$. We denote \\
										&						& $N_G(u)$ by $N(u)$ when no confusion will arise.\\\hline
The \emph{degree} of $u$						& $d(u)$					& The cardinality of $N(u)$.\\\hline
The \emph{neighbourhood} of $S$			& $N_G(S)$				& The set of vertices $v\in V(G)-S$ such that $u\sim v$\\
																		&									& for some $u\in S$.\\\hline
$H$ is a \emph{subgraph} of $G$			& $H\subseteq G$ 	& $H$ is a graph such that $V(H)\subseteq V(G)$ and\\
																		&									& $E(H)\subseteq E(G)$.\\\hline
The subgraph of $G$ 							 	& $G[S]$					& The graph with vertex set $S$, containing all\\	
\emph{induced} by $S$								&									& edges of $G$ between members of $S$.\\\hline
$H$ is an \emph{induced} subgraph		&									& There is a set $S\subseteq V(G)$ such that $G[S]=H$. \\
of $G$															&									& \\\hline
The \emph{complement} of $G$				& $\overline{G}$ 	& The graph on vertex set $V(G)$ where vertices $u$\\
																		&									& and $v$ in $V(G)$ are adjacent in $\overline{G}$ if and only if\\
																		&									& they are not adjacent in $G$.\\\hline
$S$ is a \emph{stable} set					& 								& No two vertices of $S$ are adjacent. \\\hline
$S$ forms a \emph{clique}						&									& Every pair of vertices in $S$ are adjacent.\\\hline
The \emph{stability number} of $G$ 	&	$\alpha(G)$			& The size of the largest stable set in $V(G)$.\\\hline
The \emph{clique number} of $G$			& $\omega(G)$			& The size of the largest set $S\subseteq V(G)$ which forms\\
																		&									& a clique.\\\hline
The \emph{maximum degree} of $G$		&	$\Delta(G)$			& The maximum of $d(u)$ over the vertices $u\in V(G)$.\\\hline
The \emph{minimum degree} of $G$		&	$\delta(G)$			& The minimum of $d(u)$ over the vertices $u\in V(G)$.\\\hline
A \emph{bipartition} of $G$					&									& A pair $(A,B)$ of stable sets which partition $V(G)$.\\\hline
$M$ is a \emph{matching}						& 								& No two edges in $M$ share an endpoint.
\end{longtable}
\end{center}

\section*{Special Graphs}

\begin{center}
\begin{longtable}{l|l|l}
\hline{\bf Terminology }						& {\bf Notation}	& {\bf Definition/Meaning } \\ \hline\hline
A \emph{complete graph} of order 				& $K_k$						& A graph consisting of a set of $k$ vertices, any two\\ 
$k$																			&									& of which are adjacent.\\\hline
$G$ is \emph{bipartite}									&									& There exists a bipartition of $G$.\\\hline
A \emph{complete bipartite} graph				&	$K_{a,b}$				& A graph consisting of a stable set of size $a$, a\\
																				&									& stable set of size $b$, and all edges between these\\
																				&									& two sets.\\\hline
A \emph{complete $k$-partite} or		 		&									& A graph consisting of $k$ non-empty sets \\
\emph{multipartite} graph								&									& $P_1,\dots,P_k$, called \emph{parts}, such that two vertices \\
																				&									&  of $G$ are adjacent precisely when they belong\\
																				&									&  to different parts.\\\hline
A \emph{claw}														&	$K_{1,3}$				& A copy of the graph $K_{1,3}$.\\\hline
$G$ is \emph{claw-free}									& 								& The claw is not an induced subgraph of $G$.\\\hline
$G$ is \emph{planar}										&									& $G$ can be drawn in $\mathbb{R}^2$ such that edges of $G$ \\
																				&									& intersect only at vertices of $G$.\\\hline
$G$ is \emph{cubic}											&									& Every vertex $u$ of $G$ has $d(u)=3$.\\\hline
The \emph{line graph} of $G$						& $L(G)$					& A graph with vertex set $E(G)$, where vertices of\\
																				&									& $L(G)$ are adjacent if they share an endpoint in \\
																				& 								& $G$.\\\hline
The \emph{total graph} of $G$						&	$T(G)$					& A graph with vertex set $V(G)\cup E(G)$, where \\
																				&									& vertices of $T(G)$ are adjacent if they are (1) \\
																				&									& adjacent vertices of $G$, (2) edges of $G$ that are\\
																				&									& adjacent in $L(G)$, or (3) an edge of $G$ and one of\\
																				&									& its endpoints.\\\hline
The \emph{square} of $G$								& $G^2$						& A graph on vertex set $V\left(G^2\right):=V(G)$, where \\
																				&									& vertices $u,v$ of $G^2$ are adjacent if either they are\\
																				& 								& adjacent or $N(u)\cap N(v)\neq \emptyset$. 
\end{longtable}
\end{center}

\section*{Graph Colouring and Choosability}

In what follows, $G$ is a graph, $C$ is a set of colours, $f$ is a function mapping $V(G)$ to $C$, $c$ is an element of $C$, and $L(v)$ is a subset of $C$ for each $v\in V(G)$.

\begin{center}
\begin{longtable}{l|l|l}
\hline{\bf Terminology }						& {\bf Notation}	& {\bf Definition/Meaning } \\ \hline\hline
$f$ is a \emph{proper (vertex)}								& 								& If $u$ and $v$ are adjacent, then \\
\emph{colouring} of $G$												&									&	$f(u)\neq f(v)$.\\\hline
$f$ is a \emph{$k$-colouring} of $G$					&									& $f$ is a proper colouring of $G$ and $|C|=k$.\\\hline
$G$ is \emph{$k$-colourable} 									& $\chi(G)\leq k$ & There exists a $k$-colouring of $G$.\\\hline
The \emph{chromatic number} of $G$						& $\chi(G)$				& The minimum $k$ such that $G$ is \\
																							&									& $k$-colourable.\\\hline
The \emph{colour class} for $c$ under $f$			& $f^{-1}(c)$			& The set of vertices which are mapped by \\
																							&									&	$f$ to a given colour $c$.\\\hline
$L$ is a \emph{$k$-list assignment} of $G$		&									&	$|L(v)|\geq k$ for all $v\in V(G)$.\\\hline
$f$ is an \emph{acceptable} colouring 				&									& $f$ is a proper colouring and $f(v)\in L(v)$ \\
for $L$																				&									& for all $v\in V(G)$.\\\hline
$G$ is \emph{$k$-choosable}										& $\ch(G)\leq k$	& There is an acceptable colouring for\\
																							&									& every $k$-list assignment $L$ of $G$.\\\hline
The \emph{choice number} of $G$								&	$\ch(G)$				& The minimum $k$ such that $G$ is\\
																							&									& $k$-choosable.
\end{longtable}
\end{center}

}

\bibHeading{References}
\bibliography{Masters}
  \bibliographystyle{alpha}
  \nocite{*}
  
\end{document}